\documentclass[12pt]{amsart}
\usepackage{enumerate,enumitem,amsmath,amssymb, mathrsfs}%,stmaryrd}

\usepackage{color} 
\usepackage{latexsym, amssymb,enumerate}
\usepackage{appendix}

\newcommand{\bpf}{\begin{proof}}
\newcommand{\epf}{\end{proof}}
\newcommand{\beq}{\begin{equation}}
\newcommand{\eeq}{\end{equation}}
\newcommand{\beqn}{\begin{eqnarray*}}
\newcommand{\eeqn}{\end{eqnarray*}}

\newcommand\tr{\mathop{\rm tr}\nolimits}

\newcommand\R{\mathbb{R}}

\makeatletter
\newcommand*{\greek}[1]{%
\ifcase#1\relax
\or A%
\or B%
\or C%
\else\@ctrerr \fi
}
\makeatother

\newtheorem{atheorem}{Theorem}

\def\circledwedge{\setbox0=\hbox{$\bigcirc$}\relax \mathbin {\hbox
to0pt{\raise.5pt\hbox to\wd0{\hfil $\wedge$\hfil}\hss}\box0 }}

\def\H{{\mathbb H}}

\def\R{{\mathfrak R}}

\newtheorem{prop}{Proposition}[section]
\newtheorem{theo}[prop]{Theorem}
\newtheorem{lemm}[prop]{Lemma}
\newtheorem{coro}[prop]{Corollary}

\newtheorem{defi}[prop]{Definition}

\newtheorem{RK}{Remark}
\def\begeq{\begin{equation}}
\def\endeq{\end{equation}}
\def\p{\partial}

\def\R{\mathbb R}

\def\tr{{\rm tr}}

\def \ds{\displaystyle}

\def\S{\mathbb  {S}}

\def\odot{\setbox0=\hbox{$\bigcirc$}\relax \mathbin {\hbox to0pt{\raise.5pt\hbox to\wd0{\hfil $\wedge$\hfil}\hss}\box0 }}

\numberwithin{equation} {section}
\def\tilde{\widetilde}

\begin{document}

\title[Conformally compact Einstein manifolds] {A Rigidity theorem on conformally compact Einstein manifolds in high dimensions}

\author{Sun-Yung Alice Chang}
\address{Department of Mathematics, Princeton University, Princeton, NJ 08544, USA}
\email{syachang@math.princeton.edu}

\author{Yuxin Ge}
\address{IMT,
Universit\'e Toulouse 3 \\118, route de Narbonne
31062 Toulouse, France}
\email{yge@math.univ-toulouse.fr}

\begin{abstract} 
In this paper, we establish a Liouville type rigidity result for a class of asymptotically hyperbolic non-compact Einstein metrics defined on manifolds of dimension $d\ge 5$  extending the earlier result in dimension $d=4$. 
\end{abstract}

\thanks{Research of Chang is supported in part by Simon Foundation Travel Fund. Research of  Ge is  supported in part by the grant ANR-23-CE40-0010-02  of the French National Research Agency (ANR): Einstein constraints: past, present, and future (EINSTEIN-PPF).}

\subjclass[2000]{}

\keywords{}

\maketitle

\begin{center}
 \it Dedicated to Professor Jean-Michel Coron in celebration of his 70th birthday   
\end{center}

\section{Introduction}\label{introduction and statement of results}

\subsection{Statement of results}

Let $X^d$ be a smooth manifold of dimension $d$ with $d \geq 3$ with boundary $\partial X$.
A smooth conformally compact metric $g^+$ on $X$ is a Riemannian metric such that $g = r^2 g^+$ extends smoothly to the 
closure $\overline{X}$ for some defining function $r$ of the boundary $\partial X$ in $X$. A defining function $r$ is a smooth 
nonnegative function on the closure $\overline{X}$ such that $\partial X = \{r=0\}$ and the differential $D r \neq 0$ on $\partial X$. A 
conformally compact metric $g^+$ on $X$ is said to be conformal compact Poincaré-Einstein (abbreviated as CCE) if, in addition, 
$$
\operatorname{Ric} [g^+] = - (d-1) g^+.
$$
The most significant feature of CCE manifolds $(X, \ g^+)$ is that the metric $g^+$ is ``canonically" associated with the conformal
structure $[\hat g]$ on the boundary at infinity $\partial X$, where $\hat g = g|_{T\partial X}$. $(\partial X, \ [\hat g])$ is called the conformal infinity of the conformally 
compact manifold $(X, \ g^+)$. It is of great interest in both the mathematics and theoretical physics communities to understand the correspondences between 
conformally compact Einstein manifolds $(X, \ g^+)$ and their conformal infinities $(\partial X, \ [\hat g])$, especially in view of the AdS/CFT
correspondence in theoretical physics (cf. Maldacena \cite{Mald-1,Mald-2,Mald} and Witten  \cite{Wi}). \\ 

The project we work on in this paper is to address the following  compactness issue:  Given a sequence
of CCE manifolds $(X^d, M^{d-1}, \{g_i^{+}\})$ with $M = \partial X$ and $\{ g_i \} = \{ r_i^2 g_i^{+} \} $ the corresponding sequence of compactified metrics, denote 
$ h_i = g_i|_{M}$ for each $i$, under the assumption that $\{h_i\}$ forms a compact family of metrics in $M$,when is it true that some 
representative metric ${\bar g_i} \in  [g_i]$
with $\{ {\bar g}_i |_M =h_i\} $ also forms a compact family of metrics in $X$?
We remark that, for any CCE manifold, given any conformal infinity, a special defining function which we call the geodesic defining function $r$ exists 
so that $ | \nabla_{\bar g} r | \equiv 1 $ in an asymptotic neighborhood $M \times [0, \epsilon)$ of $M$.     
We also remark that the eventual goal of the study of the compactness issue is to establish some existence results of conformal filling in for some given classes of
Riemannian manifolds as the conformal infinity,  a goal we have only partially achieved when dimension $d =4$ \cite{CG}.
\\

One of the difficulties in addressing the compactness problem is due to the existence of some ``non-local" term in the asymptotic expansion of the metric near 
the conformal infinity. To see this, when $d$ is even, we look at  
 the asymptotic behavior of the compactified metric $g$ of CCE manifold $(X^{d} ,M^{d-1}, g^{+})$ with conformal infinity $(M^{d-1}, [h] )$ (\cite{G00}, \cite{FG12}) which  takes the form
\begin{equation}
\label{expansioneven} 
 g:=r^2 g^{+} = h + g^{(2)} {r}^2 +\cdots(\mbox{even powers})+ g^{(d-1)} r^{d-1} + g^{(d)} {r}^{d} + \cdot \cdot \cdot \cdot 
\end{equation}
on an asymptotic neighborhood of $M \times (0, \epsilon)$, where $r$ denotes the geodesic defining function of $g$. The $g^{(j)}$ are tensors on $M$, and $g^{(d-1)}$ is trace-free with respect
to a metric in the conformal class on $M$. For $j$ even and $0 \le j \le
d-2$, the tensor $g^{(j)}$ is locally formally determined by the conformal representative, but $g^{(d-1)}$ is a non-local term which is not determined by the boundary metric $h$, subject to the trace free condition.  

When $d$ is  odd, the analogous expansion is
\begin{equation}
\label{expansionodd} 
 g:=r^2 g^{+} = h + g^{(2)} {r}^2 +\cdots(\mbox{even powers})+ g^{(d-1)} r^{d-1} + k {r}^{d-1} \log r+ \cdot \cdot \cdot \cdot 
\end{equation}
where now the $g^{(j)}$ are locally determined for $j $ even and $0 \le  j \le 
d-2$, $k$ is locally determined and trace-free, the trace of  $g^{(d-1)}$ is locally
determined, but the trace-free part of $g^{(d-1)}$  is formally undetermined. 
We remark that $h $ together with $g^{(d-1)}$ determine the asymptotic behavior of $g$ (\cite {FG12}, \cite{Biquard}). \\

A model case of a CCE manifold is the hyperbolic ball $\mathbb{B}^d$ equipped with the Poincar\'e metric
$g_{H}$ and  with the conformal infinity being the standard canonical metric
$d-1$ sphere $\mathbb{S}^{d-1}$. In this case, it was proved
by \cite{Q} (see also \cite{Dutta} and later the proof by \cite{LQS}) that $(\mathbb{B}^d, g_{\mathbb{H}} )$ is the unique CCE manifold with the standard metric on $\mathbb{S}^{d-1}$ as its conformal infinity.

Another class of examples of CCE manifolds is the perturbation existence result of Graham-Lee \cite{GL} in 1991, where they established that for all $d\geq 4$, a metric on $\mathbb{S}^{d-1}$ close enough in $C^{2, \alpha}$ norm to the standard metric is the conformal infinity of some CCE metric
on the ball $\mathbb{B}^d$. We remark that in an earlier paper \cite{CGQ}, when the dimension $d$ is four, we have established also 
the uniqueness of the CCE extension
of Graham and Lee under the slightly stronger assumption $C^{3, \alpha}$ condition. The same perturbation uniqueness problem was later generalized to all dimensions $d \ge 5$ in \cite{CGJQ} as a consequence of the perturbation compactness result.
 
%The goal of this paper to extend the above result in
%\cite{CGQ} to all dimensions $d\ge 4$. 

%In \cite{CG} and \cite{CGQ}, we have considered a special choice of compactification $g^{*}=e^{2w}g^+$ on a CCE manifold $(X^4, M^3, g^{+})$ of dimension four,  which we named as the Fefferman-Graham's (FG) compactification, defined
%by solving the PDE:
%\begin{equation}
%\label{fgbis}
%- \Delta_{g^{+}} w = 3  \,\,\,\, on  \,\,\, X^4.
%\end{equation}

In all  \cite{CG1} \cite {CGQ},\cite{CGJQ} and \cite{CG}, we have chosen a special representative metric among the compactified metrics, which have been studied earlier in the paper Case-Chang, \cite{casechang}, and named in general as "adapted metrics". In the case of a $d$-dimensional CCE manifold $(X^{d}, M^{d-1}, g^{+})$ when $d>4$ with conformal infinity $(M^{d-1}, h)$,  
the adapted metric which we choose in this article is defined by solving the PDE:
\begin{equation}
\label{fg}
- \Delta_{g^{+}} v-\frac{(d-1)^2-9}{4} v= 0  \,\,\,\, on  \,\,\, X^d,
\end{equation}
The adapted metric is defined as the metric  $g^{*} : = v^{\frac {4}{d-4}} g^+ $   with $g^{*} |_M = h$. We remark that one of the most important properties of the 
adapted metric is that it has vanishing Q-curvature.
\\
%We recall a local Yamabe constant on complete non-compact $(X,g)$ defined as follows:\\
%\begin{equation}
%\label{Yamabe}
%Y(X,g)=\inf_{u\in C^\infty_0(X)\setminus\{0\}} \frac{\int_X(\frac{4(d-1)}{d-2}|\nabla u|^2+R[g]u^2)dvol[g]}{(\int_X u^{\frac{2d}{d-2}}dvol[g])^{\frac{d-2}{d}}}
%\end{equation}
%When $X$ is compact, it is just standard Yamabe constant in the conformal class. We remark
%$$
%Y(\mathbb{R}^{d}, g_E)=Y(\mathbb{H}^{d}, g_H)= Y(\mathbb{S}^{d}, [g_S])
%$$
%where $g_E$ (resp. $g_H$ or, $g_S$) is the standard euclidean (resp. hyperbolic, spherical) metric. 
%To get the compactness result by blow up analysis, a key point is to prove a Liouville type theorem. Let us denote $g_E$ (resp. $g_H$ or, $g_S$) the standard euclidean (resp. hyperbolic, spherical) metric.  The main result we have is:

WE now state the main result of this paper.

\begin{theo}\label{Liouville}
Let $(X,g^+)$ be a $C^{2,\alpha}$ conformally Poincare Einstein metric of dimension $d\ge 5$ with its adapted metric $g$ a complete non-compact metric on the boundary $\p X$ with
$(\p X,g|_{T\p X})=(\mathbb{R}^{d-1}, g_{\mathbb{R}^{d-1}})$.
Assume further $g$ satisfying:
\begin{enumerate}
%%\item The defining function $\rho$ with $g_\infty=\rho^2g_\infty^+$ satisfies $|\nabla[g_\infty]\rho|\le 1$.
%\item $(\p X,g|_{T\p X})=(\mathbb{R}^{d-1}, g_{\mathbb{R}^{d-1}})$. 

%\item The interior and boundary injectivity radius of $g$ being positive and the scalar curvature of $g$ is non-negative.
% \item $(\p X,g|_{T\p X})=(\mathbb{R}^{d-1}, g_{\mathbb{R}^{d-1}})$. 
%\item  $g_\infty$ has the non-negative scalar curvature, free $Q_4$-curvature and bounded Riemann curvature.
\item The scalar curvature of $g$ is non-negative and the Riemann curvature of g is bounded $\|Rm_.g\|_{C^0}\le 1$.
\item For some fixed point $O\in\bar X$, we have 
$$
\limsup_{x\to\infty} |Ric(x)| (1+d_g(x,O))^2:=\varepsilon_1<\varepsilon<\frac{1}{2}
$$
\item The interior and boundary injectivity radius of $g$ are positive.

%$g_\infty=dx^2+O(\rho^2)$ in the neighborhood of $\p X_\infty$, that is, $\rho$ is almost geodesic defining function of the order $2$, where $x=(\rho, x_2,\cdots,x_d)$ and $dx^2$ is the euclidien metric.

\end{enumerate}
We conclude that $g^+$ is the hyperbolic space and $g$ is the flat 
Euclidean metric $g_E$.
\end{theo}

We now explain our motivation in establishing the rigidity result above.  For this purpose, we 
will first recall the program in dimension 4 and an analogous Liouvulle type rigidity result which was derived earlier in \cite{CG} when in the case when dimension $d=4$.

We start with another special choice of "adapted metric" as in \ref{fg}
when $d=4$, which was named as the  "adapted scalar flat metric" in \cite {CG}, and has the special property that the scalar curvature of the metric is identically zero, the rationale for such a choice is explained in section 2.3 of the manuscript \cite{CG} as well.
\vskip .1in
\begin{defi} 
On $(X^4, M^3,g^+)$ a CCE manifold of dimension 4, we consider a special case of the  Poisson equation with the Dirichlet data  $ f \equiv 1$,  
\beq
\label{Poisson}
-\triangle_+ v  -  2 v =0, 
\eeq
and define the compactification of $g^+$
$$
g= \rho^2 g_+ \,\,\,\,\, where \,\,\,  \rho = v.
$$
We call such a choice of $g$ a {scalar flat adapted metric}.
\end{defi}

\vskip .1in

We now state an analog of the Liouville type rigidity result in dimension 4.

\vskip .1in

\noindent %{\underline {Theorem}} [\cite{CG}, Theorem 2]\\
\begin{atheorem}(\cite[Theorem 2]{CG})
\label{aTheorem Liouville}
Let $(X_\infty,g_\infty^+)$ be $4$ dimensional $C^{2, \alpha} $ Poincar\'e Einstein metric with conformal infinity $(\R^3, dy^2)$. Denote $g_{\infty}$ the corresponding  adapted scalar flat metric. Assume $g_{\infty}= \rho^2 g_{\infty}^+$  satisfies the following conditions:
\begin{enumerate}
%%\item The defining function $\rho$ with $g_\infty=\rho^2g_\infty^+$ satisfies $|\nabla[g_\infty]\rho|\le 1$.
% \item The interior and boundary injectivity radius of $g_\infty$ are positive.
%\item $(\p X_\infty,g_\infty|_{T\p X_\infty})=(\mathbb{R}^{3}, g_{\mathbb{R}^{3}})$. 
%\item  $g_\infty$ has the non-negative scalar curvature, free $Q_4$-curvature and bounded Riemann curvature.
\item  $g_\infty$  has vanishing scalar curvature and bounded Riemann curvature with $|\nabla \rho|_{g_{\infty}} \le 1$.
\item There exists some $\varepsilon\in (0,\frac{3}{16})$ such that
$$
\lim_{a\to\infty} \frac{\inf_{y\in \partial B_{\R^3}(O,a)} d(y,O)}{a^{\frac{1}{1+\varepsilon}}}=+\infty
$$
\item There exists some positive constants $\delta\in (\frac{16\varepsilon}{3},1),C>0$ such that  for some fixed point $p$ we have 
$$|Ric[g](y)|\le C(1+d(y,p))^{-\delta}.$$
%$g_\infty=dx^2+O(\rho^2)$ in the neighborhood of $\p X_\infty$, that is, $\rho$ is almost geodesic defining function of the order $2$, where $x=(\rho, x_2,\cdots,x_d)$ and $dx^2$ is the euclidien metric.
\item The interior and boundary injectivity radius of $g_\infty$ are positive.
\end{enumerate}
Then $g_{\infty}^+ $ is the hyperbolic metric, and the compactified metric $g_\infty$ is the flat metric on ${\mathbb R}^4_{+}$
\end{atheorem}
\vskip .2in

In \cite{CG}, we have applied the hyperbolic rigidity result above to derive the following compactness result.

\begin{atheorem} 
\label{maintheorem}
Let $\{X,M=\p X, g_i^+\}$ be a family of $4$-dimensional oriented conformally compact Einstein metrics on X with conformal infinity $M$. Assume  
\begin{enumerate}

\item Suppose $\{{\hat g}_i\}$ is a sequence of metrics on M, and there exists some positive constant $C_0 > 0$ such that the Yamabe constant on the boundary is bounded below by $C_0$, i.e.
$$
Y(M,[{\hat g}_i]) \ge C_0. 
$$
Suppose also that the set 
$\{\hat g_i \}$ is compact in the $C^{k,\alpha}$ norm with $k\ge 6$ and has  positive scalar curvature.
\item There exists some positive constant $C_1 > 0$ such that
$$
\int |W_{g_i}|^2 \le C_1,
$$

\item  $H_2(X,\mathbb{Z})=0$ and $H^2(X,\mathbb{Z})=0$.
\end{enumerate}
Then the family of adapted scalar flat metrics $(X,g_i)$ is compact in the $C^{k,\alpha'}$ norm for any $\alpha'\in (0,\alpha)$ up to a diffeomorphism fixing the boundary.
\end{atheorem}

The method of proof in dimension 4 case in \cite{CG} to establish Theorem \ref{maintheorem} from Theorem \ref{aTheorem Liouville} is via the following contradiction argument.
A first step is to show that under conditions (1) and (2) of theorem \ref{maintheorem}, the adapted scalar flat  metrics $\{ g_i \}$ have a uniform $C^{3} (X)$ bound. Assume the contrary, we then study the blow-up
Gromov-Hausdorff limit $g_{\infty}$ of some suitably re-scaled metrics of $g_i$, and verify
that $g_{\infty}$
satisfies all the conditions listed in the statement of Theorem \ref{aTheorem Liouville}. The crucial step in the proof of \ref{maintheorem} then reduces to establishing Theorem \ref{aTheorem Liouville}.

We remark that in trying to push the program above from CCE manifolds of dimension $d=4$ to higher dimensions $d$, we are still searching for suitable conditions to replace the curvature condition (2) that Weyl curvature be in $L^2$ which appears in the statement of Theorem \ref{maintheorem}, so that we may possibly apply the contraction argument to show its blow-up limit would satisfy the 
conditions listed in Theorem \ref{Liouville} in the current paper.
An additional difficulty is to find some suitable topological 
conditions to replace the condition (3) in the statement of Theorem \ref{maintheorem}, which was used to prevent the interior blow- 
up of the sequence of compactified metrics.

We also remark that as an application of the compactness theorem \ref{maintheorem}, In \cite{CG}, we have also established some existence results for CCE manifolds of dimension 4 with conformal infinity the three sphere with metrics which are not perturbation 
results of the standard canonical metric.

In a recent preprint \cite{LeeWang}, Lee and Wang showed a rigidity result for the hyperbolic metric similar in spirit to our Theorem \ref{Liouville}. Their result holds for any $C^{3,\alpha}$ CCE compactified metric in all dimensions $d\ge 4$ but under conditions stronger than the assumption (2) in the statement of Theorem \ref{Liouville}. 

%for full Riemann curvature tensor holds everywhere and for some non-explicit sufficiently small number $\varepsilon_1$, namely $|Rm(x)| d_g(x,O)^2<\varepsilon_1(n)$ for all $x\in \bar X$ and for some small constant $\varepsilon_1(n)$. In this work, we assume only the condition (2) holds for the Ricci curvature at the infinity with a precise bound $\varepsilon_1<\frac{1}{2}$. However, we study the dimension of manifold $d\ge 5$. The case $d=4$ is more delicate and has been done in the preprint \cite{CG}.

The paper is organized as follows: In section \ref{Sect:prelim}, we recall some basic ingredients about  the adapted metrics.   In section \ref{Sect:Liouville}, we prove the main result the Liouville  Theorem \ref{Liouville}.\\

%%%%%%%%%%%%%%%%%%%%%%%%%%%%%%%%%%%%%%%%%%%%%%%%%%%%%%

\section{Preliminaries}\label{Sect:prelim}
\subsection{Adapted metrics $g^*$}\label{Subsect:rho}

We now consider a class of adapted metrics $g^*$ by solving (\ref{fg}) to study the compactness problem of CCE manifolds $(X^d,M^{d-1}, g^+ )$ of dimension d greater than 4.

\begin{lemm} (Chang-R. Yang \cite{CYa})\label{fgmetric}
Suppose $(X^{d}, {\p X}, g^{+})$ is conformally compact Einstein with conformal infinity $({\p X}, [h])$,  fix $h_1 \in [h ]$ and $r$ its
corresponding geodesic defining function.  Consider the solution  $v$ to (\ref{fg}),
then $ v $ has the asymptotic behavior 
$$ v=  \, r^{\frac{d-4}{2}}(A + B r^3) $$ near ${\p X}$, 
where $A, B$ are functions even in $r$, such that $A|_{\p X}\equiv 1$. Moreover, 
consider the metric $g^{*} = v^{\frac4{d-4}} g^+ $, then $g^*$ is totally geodesic on boundary with 
the free $Q$-curvature, that is,  $Q_{g^{*}} \equiv 0 .$  
%(2)  $ B|_{\p X} = \frac{1}{36} \frac{\partial}{\partial n} R_{g^*} = \frac {1}{3} T_{g^{*}} .$
\end{lemm}

Suppose that $X$ is a smooth $d$-dimensional manifold with boundary $\partial X$ and $g^+$ is a conformally compact 
Einstein metric on $X$. Denote $g^*=\rho^{2}g^+$ the adapted metrics, that is, the corresponding $v:= \rho^{\frac{d-4}{2}}$ satisfies the equation (\ref{fg}). 

To derive the geometric properties of $g^*$, we first recall some basic formulas of curvatures under conformal changes. Write 
$g^+ = r^{-2} g$ for some defining function $r$ and calculate
$$
Ric[g^+] = Ric[g] + (d-2)r^{-1} \nabla^2 r + (r^{-1} \triangle r - (d-1)r^{-2} |\nabla r|^2) g.
$$
As a consequence we have
$$
R[g^+] = r^2 (R[g] + \frac{2d-2}{r} \triangle r -\frac{d(d-1)}{r^2}  |\nabla r |^2).
$$
Here the covariant derivatives are calculated with respect to the metric $g$ (or the adapted metrics $g^*$ later in the section). Thus for the adapted metrics  $g^*$ which is conformal to the  Einstein metric $g^+$, one has 
\beq \label{relation3}
R[g^*] = 2(d-1)\rho^{-2}(1-|\nabla\rho|^2),
\eeq
which in turn implies
\beq
\label{relation5}
Ric[g^*]=- (d-2)\rho^{-1} \nabla^2 \rho+\frac{4-d}{4(d-1)}R[g^*]g^*
\eeq
and
\beq
\label{relation4}
R[g^*] = - \frac{4(d-1)}{d+2}\rho^{-1} \triangle \rho.
\eeq

We now recall 

\begin{lemm} (\cite{casechang}), \cite[Lemma 4.2]{CG})
\label{lem4.1}
Suppose that $X$ is a smooth $d$-dimensional manifold with boundary $\partial X$ and $g^+$ be a conformally compact Einstein metric on 
$X$ with the conformal infinity $(\p X, [h])$ of nonnegative Yamabe type. Denote $g^*=\rho^2g^+$ be adapted metrics associated with the  metric $h$ with the positive scalar curvature in  the conformal infinity. Then the scalar curvature $R[g^*]$ is positive in 
$X$. In particular, 
\beq\label{estimate-rho}
\|\nabla\rho \| [g^*] \le 1.
\eeq
\end{lemm}

\section{Louville type rigidity result}\label{Sect:Liouville} 

Our first step to establish Theorem \cite{CG}  is to compare the distance between two points on the boundary to their euclidean distance on the boundary. We state our result as follows:
\begin{theo}
\label{curvaturedecay}
%Let $(X_{\infty}, g_{\infty})$ be a $C^5$ complete non-compact metric  satisfying  properties (1)-(6) in Lemma \ref{limiting-metric}.   
%with the boundary related to some compactification of CCE metric.  Assume 
%\begin{enumerate}
%\item $g$ is Bach flat and with vanishing  scalar curvature;
%\item  the boundary $M=\p X$ is umbilic and $M$ is flat; 
%\item $\ds\int_X |Rm|^2<\infty$;
%item $H\ge 0$ and $\ds\int_{\p X} H^3<\infty$;
%\item The Sobolev inequalities (\ref{sobolev1}) and  (\ref{sobolev2}) hold for all $f\in H^1(\bar X)$ with compact support 
%\end{enumerate} 
Under the same assumptions as in Theorem \ref{Liouville}, for any fixed point $O \in \partial X_{\infty}$, and any for all $y\in \R^{d-1}$
$$
d_g(y, O)\le |y|,
$$
where $|y|$ denotes the Euclidean distance of $y$ to the  point $O$. 
Moreover, we have 
\beq 
\label{conformalfactorestimate1}
\lim_{a\to\infty} \frac{\inf_{y\in \partial B_{\R^{d-1}}(O,a)} d_{g} (y, O)}{a^{\frac{1}{1+\frac{\varepsilon}{2}}}}=+\infty
\eeq
where  $B_{\R^{d-1}}(O,a)$ is the Euclidean ball of center $O$ and radius equals to $a$ on  $\R^{d-1}$.
\end{theo}

\begin{proof}

It is clear that $d(y,O)\le |y|$ when $y\in\R^{d-1}$ and $O$ is the origin on $\R^{d-1}$. To establish the last statement of Theorem \ref{curvaturedecay}, we divide the proof into 3 steps.\\
We remark 
if $O\in\R^{d-1}$, then $\rho(y)/d(y,O)\le 1$ for all $y\in X$ since $|\nabla \rho|\le 1$. Thus, we obtain 
$$
0\le 1-|\nabla \rho(x)|^2=\frac{R(x)\rho(x)^2}{2(d-1)}\le \frac{|Ric(x)|\sqrt{d}\rho(x)^2}{2(d-1)}\le \frac{|Ric(x)|\rho(x)^2}{2}
$$
We choose $\epsilon\in (\varepsilon_1,\varepsilon)$. Let $D>0$ be a sufficiently large  number such that $|Ric(y)|\le \frac{\epsilon}{(1+d(O,y))^2}$ for all $y\in X\setminus B(O, D)$. Thus, we have for all $y\in X\setminus B(O, D)$ and all tangent vector $\xi$
$$
 1-|\nabla \rho(x)|^2\le \frac{\epsilon\rho(x)^2}{2(1+d(y,O))^2}<\frac{1}{4}
$$
and
$$
|\nabla^{(2)}\rho(y)(\xi,\xi)|\le  \frac{\rho(y)}{d-2}(|Ric(y)|+\frac{d-4}{4(d-1)}R(x))|\xi|^2\le 
\frac{\rho(y)|Ric(y)|}{d-2}(1+\frac{(d-4)\sqrt{d}}{4(d-1)})|\xi|^2
$$
We consider the unit vector field $\frac{\nabla \rho}{|\nabla \rho|}$. Let $\varphi (s)$ be the flow of the field, that is $\partial_s \varphi (s)= \frac{\nabla \rho}{|\nabla \rho|}(\varphi (s))$ with the initial condition on the boundary $\varphi (0,\cdot)\in \R^{d-1}$.\\

{\sl Step 1. We claim there exists a large number $D_1>D$ such that $\varphi ([0,\infty)\times (\R^{d-1}\setminus B_{\R^{d-1}}(0,D_1)))$ is a smooth embedding into $X\setminus B(O,D)$}.\\

For this purpose, it is clear that
$$
\sup_{y\in \bar B(O, D)} \rho(y)\le D
$$
since $|\nabla \rho|\le 1$. We can choose $D_1>0$ such that 
$$
\bar B(O, 4D)\cap \R^{d-1}\subset B_{\R^{d-1}}(O,D_1)
$$
For all $(s, z)\in [0,D]\times \R^{d-1}\setminus B_{\R^{d-1}}(0,D_1)$, we have $\rho(\varphi(s,z))=s$ and the length of curve $\varphi([0,s]\times\{z\})$ is smaller than $\frac{s}{1-\epsilon}\le 2D$. Thus $d(\varphi(s,z), O)\ge d(\varphi(0,z), O)-d(\varphi(s,z), \varphi(0,z))\ge D_1-2D\ge 2D$. On the other hand, $\rho(\varphi(\cdot,y))$ is increasing along the flow. Hence, the flow exists for all time and $\rho(\varphi ([2D,\infty)\times (\R^{d-1}\setminus B_{\R^{d-1}}(0,D_1))))\subset [2D, +\infty)$, which implies
$$
\varphi ([2D,\infty)\times (\R^{d-1}\setminus B_{\R^{d-1}}(0,D_1)))\subset X\setminus B(O,D)
$$
Thus the claim follows from the observation that the vector field is unit one on $X\setminus B(O,D)$.

We take the coordinates $[0,\infty)\times (\R^{d-1}\setminus B_{\R^{d-1}}(0,D_1))$ for $\varphi ([0,\infty)\times (\R^{d-1}\setminus B_{\R^{d-1}}(0,D_1)))$ and the metric on this set can be written as
$$
g=\frac{d\rho^2}{|\nabla \rho|^2}+ g_\rho
$$
where $g_\rho$ is a metric on $\varphi (\{\rho\}\times (\R^{d-1}\setminus B_{\R^{d-1}}(0,D_1)))$, or equivalently a family metric on $\R^{d-1}\setminus B_{\R^{d-1}}(0,D_1)$. \\

{\sl Step 2. For all $s\ge 0$, we have $(s+1)^{-\epsilon}g_{\R^{d-1}}\le g_s\le  (s+1)^{\epsilon}g_{\R^{d-1}}
$}.\\

Let $\psi :[0,1]\to \R^{d-1}\setminus B_{\R^{d-1}}(0,D_1)$ be a smooth curve on the boundary. We consider 
$$
\Psi(s,v)=\varphi(s, \psi(v))
$$
and
$$
l(s)=\int_0^1 \langle \partial_v \Psi(s,v), \partial_v \Psi(s,v)\rangle
$$
Thus, a direct calculation leads to
\begin{align*}
l'(s)&=  2  \int_0^1\langle \nabla_s \nabla_v \Psi(s,v), \nabla_v \Psi(s,v)\rangle= 2 \int_0^1 \langle \nabla_v \nabla_s \Psi(s,v), \nabla_v \Psi(s,v)\rangle\\
&=2  \int_0^1 \langle \nabla_v \frac{\nabla \rho}{|\nabla \rho|}(\Psi(s,v)), \nabla_v \Psi(s,v)\rangle\\
&=2  \int_0^1 \frac{1}{|\nabla \rho|} \nabla^{(2)}\rho( \nabla_v \Psi(s,v), \nabla_v \Psi(s,v))\\
&=  \frac{2}{d-2}\int_0^1 -\frac{s}{|\nabla \rho|} (Ric( \nabla_v \Psi(s,v), \nabla_v \Psi(s,v))+ \frac{(4-d)R(x)}{4(d-1)} |\nabla_v \Psi(s,v)|^2)
\end{align*}
since $\rho(\Psi(s,v))=s$. Hence, when $s\ge 0$, we have
$$
|l'(s)|\le \frac{\epsilon}{s+1}l(s)
$$
since $|\nabla \rho|^2\ge \frac{3}{4}$ and $1+\frac{(d-4)\sqrt{d}}{4(d-1)}\le \frac{\sqrt{3}}{4}(d-2)$ provided $d\ge 5$. 
Therefore, for all $s\ge 0$
$$
l(0)^2 (s+1)^{-\epsilon}  \le l(s)^2\le l(0)^2  (s+1)^{\epsilon} 
$$
Finally, the claim follows.\\

{\sl Step 3.  There holds (\ref{conformalfactorestimate1}).}\\

Given $y\in \partial B_{\R^{d-1}}(O,a)$ for sufficiently large $a>0$, let $\sigma:[0,1]\to X$ be a smooth curve joining $y=\sigma(0)$ to $O=\sigma(1)$. Let $s_{\max}\in [0,1)$ such that 
$\sigma([0, s_{\max}]) \subset \varphi ([0,\infty)\times (\R^{d-1}\setminus B_{\R^{d-1}}(0,D_1)))$ and $  \sigma(s_{\max})\in \partial \varphi ([0,\infty)\times (\R^{d-1}\setminus B_{\R^{d-1}}(0,D_1)))$. We use the coordinates $[0,\infty)\times (\R^{d-1}\setminus B_{\R^{d-1}}(0,D_1))$ and write $\sigma(t)=(\rho(t), \tilde y(t))$ for $t\in [0, s_{\max}]$. If $\max \rho(\sigma[0,1])\ge a^{\frac{1}{1+\epsilon/2}}$, it follows from the fact $|\nabla \rho|\le 1$ that the length of the curve $\sigma$ is bigger than $a^{\frac{1}{1+\epsilon/2}}$. Otherwise, we consider the curve 
$\sigma([0, s_{\max}])\subset \varphi([0, a^{\frac{1}{1+\epsilon/2}}]\times (\R^{d-1}\setminus B_{\R^{d-1}}(0,D_1)))$. It is clear that on $\varphi([0, a^{\frac{1}{1+\epsilon/2}}]\times (\R^{d-1}\setminus B_{\R^{d-1}}(0,D_1)))$, we have
$$
g\ge (a^{\frac{1}{1+\epsilon/2}}+1)^{-\epsilon}g_{\R^3}
$$
Therefore, the length of $\sigma([0, s_{\max}])$ is bigger than
$(a^{\frac{1}{1+\epsilon/2}}+1)^{-{\epsilon}/2}(a-D_1)$. Hence
$$
d(O,y)\ge \min\{  (a^{\frac{1}{1+\epsilon/2}}+1)^{-{\epsilon/2}}(a-D_1), a^{\frac{1}{1+\epsilon/2}}\}\ge \frac{1}{2}a^{\frac{1}{1+\epsilon/2}},
$$
provided $a$ is sufficiently large. Therefore, the desired result (\ref{conformalfactorestimate1}) follows. Finally, we prove Theorem \ref{curvaturedecay}.

\end{proof}
%Denote $g_j= K_j^2 g_j^{*}$, %$g_{\infty}$ the G-H limit %metric of $g_j$ as in Chang-%Ge-Qing. 
%\vskip .1in

We are ready  to prove  Theorem \ref{Liouville}.\\

%Suppose $g_{\infty}$ is the blow-up limit of a sequence of "generalized" FG metric $g_j*$ defined on $(X^d, S^{d-1})$ as in the proof of the main theorem in CG, CGQ and CGJQ; i.e. Denote $g_j = K_j^2 g_j^*$ with $K_j$ chosen so that $|R_m(g_j)|_{\infty} \leq 1$, and $g_{\infty}$ is the G-H limit of the sequence $g_j$ then 
%by \cite[Lemmas 3.1 and 3.3]{CGQ} and  \cite[Lemmas 2.9 and 2.11]{CGJQ}, the sequence $g_j$ is non-collapsing, i.e. $vol(B_{g_\infty}(p,1))\ge C>0$ for some point $p\in X_{\infty}$; this together with the fact that the boundary $\p X_\infty$ is totally geodesic implies that the injectivity radius of the sequence $g_j$ hence $g_{\infty}$ has a positive lower bound. This is the model case we are looking at which satisfies the condition (1) of the theorem.
{\bf Proof of Theorem \ref{Liouville}}\\   

\vskip .1in

{\it Step 1.}  We set $X_{\delta}:=\p X\times [0,\delta]$. We recall adapted compactification is %As in \cite{CDLS},  there exists 
an almost $C^2$ geodesic compactification. % which is $C^3$ compatible with the initial compactification.  
That is, the defining function $\rho$ is $C^3$ (even $C^{3,\alpha}$) in the neighborhood of the infinity $X_{\delta}$,  $\rho^2 g^+$ is close to the product metric up to order 2, 
\beq
\rho^2 g^+=d\rho^2+ \hat{g}+ O(\rho^2).
\eeq
%{\color{red} This result is proved in \cite{CDLS}}??. 

We assume $\rho^2 g^+$ is $C^{2,\alpha}$ compactification. Without loss of generality, we assume its injectivity radius are bigger than $1$ and  $\delta>1$.\\

Throughout the note, 
we fix a point $p\in X$ with $d_{\bar g}(p,\p X)=1$ with $\bar g=g$ and  $p=(1, 0,\cdots,0)$. Let $t(x)=d_{g^+}(x,p)$  and $r(x)=-\log \rho(x)/2$ (in $X_{\delta}$). Denote $ u (x) = r(x) - t(x).$ \,  Denote $\bar g=4e^{-2r}g^+$ and $\tilde g=4e^{-2t}g^+=\Psi^{\frac{4}{d-3}}\bar g$ where $\Psi=e^{\frac{d-3}{2}(r-t)}=e^{\frac{d-3}{2}u}$.  

\vskip .2in

Let $\Gamma_t:=\{x| t(x)=t\}$ and $\Sigma_r:=\{x| r(x)=r\}$. \\
We denote the induced metrics by
\beq
\bar g_r=\bar g|_{T\Sigma_r}, \bar g_t=\bar g|_{T\Gamma_t}, \tilde g_r=\tilde g|_{T\Sigma_r}, \tilde g_t=\tilde g|_{T\Gamma_t}, 
\eeq
and $\nabla_+$ (resp. $\bar \nabla$) be the connection w.r.t. $g^+$ (resp. $\bar g=g$). \\

We now define and consider
\beq
\phi(t)= g^+(\nabla_+ r, \nabla_+ t),
\eeq

%%%%%%%%%%%%%%%%%%%%%%%%%%%
We remark that although $\rho$ is $C^{2, \alpha}$ on X, the function t may be only
Lipschitz at points $x$ which is are cut locus points; thus t is not differentiable in general. In the steps 2, 3 below we will first do the formal computation at points $x\in X\setminus C_p$ out of the cut-locus $C_p$   w.r.t. to $p$ where the function $t$ is 
$C^2$; we will then address those points $x\in C_p$ which are cut locus points at the end of step 3 and before step 4
of the proof.
%%%%%%%%%%%%%%%%%%%%%%%%%%%%%
% Yuxin: Is my understanding above correct?
%%%%%%%%%%%%%%%%%%%%%%%%%

we claim at a point where 
r is $C^2$, we have the following computation:
\beq
\label{eq2.4}
\phi'(s)=\nabla_+^2 r(\nabla_+ t, \nabla_+ t)
\eeq
To see (\ref{eq2.4}), we notice $\nabla_{\nabla_+ t}  g^+(\nabla_+ r, \nabla_+ t)=  g^+(\nabla_{\nabla_+ t} \nabla_+ r, \nabla_+ t)+ g^+(\nabla_+ r, \nabla_{\nabla_+ t}\nabla_+ t)= g^+(\nabla_{\nabla_+ t} \nabla_+ r, \nabla_+ t)=\nabla_+^2 r(\nabla_+ t, \nabla_+ t)$. \\

%{\it Step 2.}  $u$ is bounded.\\
%For this purpose, we fix $r_0$ large. Given  $r>r_0$, let $x\in \Sigma_r$. As we are in almost $C^2$ geodesic compactification, we have $|d_{g^+}(x, \Sigma_{r_0})-(r-r_0)|\le %C$ where $C$ is some positive constant independent of $r$. Therefore $|d_{g^+}(x, p)-r|\le (C+r_0+\mbox{diam}({\Sigma_{r_0}})$. The desired claim follows.\\

\begin{lemm}
In $X\setminus C_p$, for all $X,Y\perp \nabla_+ r$,  
\beq
\label{Lemma2.2}
\nabla_+^2 r(X,Y) =g^+(X,Y)+ O(e^{-2r}),
\eeq
\beq
\nabla_+^2 r(\nabla_+ r, X)=  O(e^{-2r}),
\eeq
\beq
\nabla_+^2 r(\nabla_+ r, \nabla_+ r)=  O(e^{-2r})
\eeq
   and
\beq
|\nabla_+ r|_{g^+}^2=1+O(e^{-2r}).
\eeq
\end{lemm}
Proof of this Lemma can be found in the paper by \cite{Dutta} and it follows from the fact we use almost $C^2$ geodesic compactification. \\

We now state some property of the function $\phi$.
Notice 
\beq
\label{eq2.9}
\nabla_+ t=g^+(\frac{\nabla_+ r}{|\nabla_+ r|_{g^+}}, \nabla_+ t)\frac{\nabla_+ r}{|\nabla_+ r|_{g^+}}+\sqrt{1-\frac{\phi(s)^2}{|\nabla_+ r|_{g^+}^2}}\nu= I( r,t) \,+ \, II(r,t),
\eeq
where 
$$ I = I(r, t) =
\phi(s)\frac{\nabla_+ r}{|\nabla_+ r|_{g^+}^2},
$$
and 
$$ II = II (r,t) = \sqrt{1-\frac{\phi(s)^2}{|\nabla_+ r|_{g^+}^2}}\nu;
$$
where $\nu $ is an unit vector $\perp \nabla_+ r$. Thus applying (2.4) and the properties in Lemma 2.2 %\ref{Lemma2.2}
, we have
\beq
\label{eq2.10}
\phi'(s)=O(e^{-2r})\phi()+(1-\phi(s)^2)(1+O(e^{-2r})).
\eeq
Hence for some small number $0<\rho_0=e^{-r_0}<1$ we have
\beq
\phi'(s)\ge -\frac18|\phi(s)|+\frac12 (1-\phi(s)^2)
\eeq
provided the geodesic $\gamma$ starting from $p$ satisfies $\gamma(t)\in  X_{\rho_0}$. \\

Thus, once $\gamma$ gets into  $X_{\rho_0}$ at first time $s_0$, we have at the point
$x_0 = \gamma(s_0)$ the angle between the vector $\gamma'(s) = \nabla_+(t)$ and 
$\nabla_+ r = - \frac{1}{\rho} \nabla_+ \rho(\gamma(s)) $ at the point $s =s_0$ is less than $\frac{\pi}{2} $, and 
$ \phi (s_0) = g^+ ( \nabla_+ r, \nabla_+ t) \geq 0$. Furthermore for all $s \geq  s_0$, on either $\phi(s) > \frac{1}{2}$ or $\phi (s)< \frac{1}{2}$ while we can see
from (2.11) that $\phi'(s) > 0$ in a neighborhood of t provided $\phi (s)< \frac{1}{2}$. Thus we conclude that for all $s\ge s_0$,  $\phi(s)\ge 0$. 
This in turn implies $\phi(\gamma (s))$ is decreasing in t when $ s \geq s_0$ and $\phi (s) \in  X_{\rho_0}$ for all $s\ge s_0$. That is, once $\gamma$ gets into $X_{\rho_0}$, it never leaves.\\

{\it Step 2.}
Apply similar arguments, if we denote $s_1$ as the first time $s_1>s_0$, 
such that $\gamma(s_1)\in \p X_{\rho_1}$, where $\rho_1 \le  {\rho_0}$ (in the interior boundary) and $\phi(s_1)= 1/2$, otherwise  we set $s_1=s_0$ and $\rho_1=\rho_0$ when $\phi(s_0)\ge 1/2$. Then 
\beq
\phi(s)\ge 1/2\; \,\,\, \forall s\ge s_1
\eeq
and any $\gamma(s)\in   X_{\rho_1}$ when $s \ge s_1.$ \\

Thus we conclude that if $\gamma(s)$ be a minimizing $g^+$ geodesic from $p$ to $x$ such that $\gamma(0)=p$. 
%Since $s_1$ is the first %time such that $%\gamma(t_1)\in \p% X_{\rho_1}$ (interior %boundary). 
For all $s\ge s_1$, we have $\phi(s)\ge 1/2$ so that $r(\gamma(s))$ is increasing function when $s\ge s_1$. We denote $x_1=\gamma(s_1)$. \\

%\end{document}
To summarize, from above argument we conclude that we have established the following Lemma.  
\begin{lemm}
\label{Lemma 2.3}
In fact, $\phi(s)$ is bounded and $\forall s\ge s_1$
\beq
\label{eq2.13}
\frac12\le \phi(s)\le |\nabla_+ r|_{g^+}\le 1.
\eeq
\end{lemm}

\begin{lemm}
\label{Lemma 2.4}
Assume $s\ge s_1$, 

\noindent (1) Then $u(\gamma(s))-u(\gamma(s_1))$ is uniformly bounded for all $x\in
  X_{\rho_1}$. \\
Moreover, we also have \\
(2) $\forall s\ge s_1$, we have $\tilde x(\gamma(s))- \tilde x(\gamma(s_1))$ is uniformly bounded for all $x\in
  X_{\rho_1}$ where $\tilde x=(x^2, \cdots, x^d)$. 
\end{lemm}
\vskip .2in

\begin{proof} We choose an adapted coordinate $(x_1=\rho,x_2,\cdots,x_d)\in [0,\delta]\times \R^{d-1}$ on $X_\delta$, where $(x_2,\cdots,x_d)$ are harmonic extensions of the euclidean coordinate on $\R^{d-1}$. In such coordinate, we have the expression for the adapted metric
$$
g_{ij}= \delta_{ij}+O(\rho^2) 
%(\mbox{ or } g_{ij}= \delta_{ij}+O(\rho^2) )
$$
where $\delta_{ij}$ is the kronecker symbol (since the boundary is totally geodesic).  We consider first $\rho(\gamma(s))$. By Lemma \ref{Lemma 2.3}, it is clear that for $s\ge s_1$
$$-\rho(\gamma(s)) \le \frac{d}{ds}\rho(\gamma(s))=- \rho(\gamma(s))  \phi(s)\le -\frac 12\rho(\gamma(s)) $$ so that $$- (s-s_1)\le \log\rho(\gamma(s))-\log\rho(\gamma(s_1))\le -\frac12 (s-s_1).$$
Similarly, for all $i=2,\cdots, d$, we have $$\frac{d}{ds}x_i(\gamma(s))= \langle \bar \nabla x_i, \gamma'(t)\rangle_{g}=O(|\gamma'(s)|_{g}|\bar \nabla x_i|_{g})=O((\rho(\gamma(s)))^2)= O(e^{-(s-s_1)})$$
which implies
$$
|x_i(\gamma(s))- x_i(\gamma(s_1))|\le C
$$
Therefore, we have established part (2)  of Lemma \ref{Lemma 2.4}.\\

We now choose $\rho_2$ and $\rho_3$, so that   $0 < \rho_2<\rho_3<\rho_1<1$, and denote $\rho_i = e^{-r_i}$ for $ 1 \leq i \leq 3.$ For any $x\in \Sigma_{r_3}$ and  $y\in \Sigma_{r_2}$, let 
$\gamma_1(s)$ be some curve connecting $x$ and $y$. We can estimate its length for the metric $g^+$
$$
\begin{array}{ll}
l(\gamma_1)&\ds=\int \sqrt{g^+(\gamma_1'(s), \gamma_1'(s))}ds=\int\frac{1}{\rho(\gamma_1(s))}\sqrt{g(\gamma_1'(s), \gamma_1'(s))}ds\\
&\ds\ge \int_{\rho_2}^{ \rho_3}\frac{(1+O(\rho))}{\rho}{d\rho}=\log\frac{ \rho_3} {\rho_2}+O( \rho_3)
\end{array}
$$
Hence $d_{g^+}(x,y)\ge  \log\frac{ \rho_3} {\rho_2}+O( 1)$. For points $x$ in the interior boundary of $\Gamma_{r_3}$, we denote $x=(\rho_3, x_2,\cdots, x_d)$. Let $\gamma_2(s)=(\rho_3-s(\rho_3-\rho_2), x_2,\cdots, x_d)$ for $s\in [0,1]$. We compute the length of $\gamma_2$ in the metric $g^+$
$$
l(\gamma_2)=\int^1_0 \sqrt{g^+(\gamma_2'(s), \gamma_2'(s))}ds=\int_{\rho_2}^{ \rho_3}\frac{(1+O(\rho^2))d\rho}{\rho}=\log\frac{ \rho_3} {\rho_2}+O(1)
$$
This infers
$$
d_{g^+}(x,\Sigma_{r_2})\le \log\frac{ \rho_3} {\rho_2}+O(1)
$$
Gathering the above estimates, we obtain 
$$
d_{g^+}(x,\Sigma_{r_2})= \log\frac{ \rho_3} {\rho_2}+O(1)
$$
for such points $x$.

Now given points $x \in X_{\rho_1}$ and a fixed point $y = \Gamma(t_1)$, consider the minimizing geodesic $\gamma$ from $\gamma(s)=x$ to $\gamma(s_1)=y$. We write $x=(\rho (x), \tilde x)$ and $y=(\rho(y) , \tilde y)$. From the above analysis, we have 
$$
d_{g^+}(x,y)\ge  \log\frac{ \rho(x)} {\rho(y)}+O( 1)
$$
On the other hand, we consider another curve $\tilde\gamma$ from $x$ to $y$ is the union of the segment from $x$ to $(\rho(x), \tilde y)$ and the segment from $(\rho(x), \tilde y)$ to $y$. We know $\tilde y-\tilde x$ is bounded.  The length of the first segment for the metric $g^+$ is $O(1)$ and that of second one is equal to $ \log\frac{ \rho_1(x)} {\rho_1(y)}+O( 1)$. Therefore, we infer
$$
l(\tilde \gamma)=\log\frac{ \rho(x)} {\rho(y)}+O( 1)
$$
which yields
$$
d_{g^+}(x,y)=  \log\frac{ \rho(x)} {\rho(y)}+O( 1)
$$
Thus we have established  part (1) of Lemma \ref{Lemma 2.4} as well, we have thus established the whole Lemma.
\end{proof}

\begin{lemm}
\label{Lemma 2.5} 
Under the assumption that $\rho^2g^+$ is some $C^{2,\alpha}$ almost geodesic compactification for some $0 < \alpha <1$, then $\forall s \geq s_1$, 
\beq
\label{eq2.13bis}
\phi'(s)+\phi(s)^2-1=O(e^{-(2+\alpha)(s-s_1)}).
\eeq
Moreover,
\beq
\label{eq2.13bisbis}
\phi(s)=1+ O(e^{-2(s-s_1)})
\eeq
Furthermore,
\beq
\label{eq2.13bisbisbis}
1-(1-\phi(s_1))e^{-2(s-s_1)}+ O(e^{-(2+\alpha)(s-s_1)})\le  \phi(s)\le 1
\eeq
\end{lemm}

\begin{proof}
To prove (\ref{eq2.13bis}) in Lemma \ref{Lemma 2.5},  we write 
\beq
\phi(s)=1+v(s)
\eeq
Apply (\ref{eq2.10})and the estimate in Lemmas \ref{Lemma 2.3} and  \ref{Lemma 2.4},
 we get 
\beq
\label{eq2.19}
v'(s)=1-(1+v)^2+O(e^{-2(s-s_1)})=-2v-v^2+O(e^{-2(s-s_1)})\ge -\frac32v+O(e^{-2(s-s_1)})
\eeq
Thus we obtain
\beq
(e^{\frac32(s-s_1)}v)'\ge O(e^{-\frac12(s-s_1)})
\eeq
so that 
\beq
v(s)\ge -Ce^{-\frac32(s-s_1)}
\eeq
Going back to (\ref{eq2.19}), this imples
\beq
v'(s)=1-(1+v)^2+O(e^{-2(s-s_1)})=-2v+O(e^{-2(s-s_1)})
\eeq
Thus, 
\beq
\label{eq2.20}
v(t)\ge -(C(s-s_1)+C)e^{-2(s-s_1)}
\eeq
On the other hand,  from (\ref{eq2.13})
$$
v(s)=\phi(s)-1\le |\nabla_+ r|_{g^+}-1\le 1-1=0
$$
so that
$$0\ge  v(s)\ge -(C(s-s_1)+C)e^{-2(s-s_1)}.$$  

We will now rewrite the expression in 
the LHS of (\ref{eq2.13bis}) into geometric quantities in the $\bar g= \rho^2 g^{+}$ metric, we will see to do so will enable us to apply the conformal invariant properties of curvature tensors to achieve our estimates. First We recall
\beq
\label{eq2.21}
Ric[\bar g]=-(d-1)g^++(d-2)(\nabla_+^2r+ \nabla_+ r\otimes  \nabla_+ r)+(\triangle_+ r -(d-2)|  \nabla_+ r|^2_{g^+})g^+
\eeq
\beq
\label{eq2.22}
R[\bar g]= \frac{1}{4}e^{2r }(-d(d-1)+ 2(d-1)\triangle_+ r+(2-d)(d-1)|  \nabla_+ r|^2_{g^+})
\eeq
combine above two formulas, we can write the Schouten tensor as
\beq
A[\bar g]=(-\frac12-\frac12|  \nabla_+ r|^2_{g^+})g^++ (\nabla_+^2 r+ \nabla_+ r\otimes  \nabla_+ r)
\eeq
Therefore, we get
\beq
\phi'(s)+\phi(s)^2-1=(-\frac12+\frac12|  \nabla_+ r|_{g+}^2)+A[\bar g]( \nabla_+ t, \nabla_+ t)
\eeq

Similarly, since $\rho=e^{-r}$, apply the expression of $A[\bar g]$ under conformal change with respect to the connection $\bar \nabla$, we get also 
\beq
A[\bar g]=-\frac{\bar \nabla^2 \rho}{\rho}+\frac {|  \bar \nabla \rho|^2_{\bar g}-1}{2\rho^2}\bar g
\eeq
%By (\ref{eq2.9}) and (\ref{eq2.20}).

Thus
\beq
\label{eqmid}
\phi'(s)+\phi(s)^2-1=-\rho^{-1}\bar\nabla^2\rho(\nabla_+ t , \nabla_+ t)+ |  \bar\nabla \rho|^2_{\bar g}-1
\eeq

We now recall formula (\ref{eq2.9}) (i.e $\nabla_+ t = I + II)$) and use the notations there to estimate the terms in (\ref{eqmid}) above.
To do so, we first claim we have
\beq
\label{crossterm}
%\begin{array}
 2\rho^{-1}\bar\nabla^2\rho( I, II) \,+\rho^{-1}\bar\nabla^2\rho( II, II)
= O((s-s_1)+1)e^{-3(s-s_1)})
%\end{array}
\eeq
 
To see the claim, we observe that as the curvature tensor of the %(rescaled adapted metric $\bar g$
 is bounded, %passing to the limit, 
 we have
\beq
\label{eq2.25bis}
\bar \nabla^2 \rho=O(\rho),\,\;\left|\frac{\nabla_+ r}{|\nabla_+ r|_{g^+}}\right|_{\bar g}=|\nu|_{\bar g}=\rho,\;|\nabla_+ r|_{g^+}=1+O(\rho^2)
\eeq
so that
\beq
\label{eq2.25bisbis}
\frac{\bar \nabla^2 \rho}{\rho}(\nu,\nu)=O(\rho^{2}), \;\frac{\bar \nabla^2 \rho}{\rho}(\frac{\nabla_+ r}{|\nabla_+ r|_{g^+}},\nu)=O(\rho^{2})
\eeq
%since $A[\bar g](\frac{\nabla r}{|\nabla r|},\nu)$ vanishes on the boundary $\p X$ and is h\"older continuous in $C^{0,\alpha}$. \\
On the other hand, we have from  (\ref{eq2.13bisbis})
\beq
\label{eq2.25bisbisbis}
\sqrt{1-\frac{\phi(s)^2}{|\nabla_+ r|^2_{g^+}}}=O(\sqrt{1+(s-s_1)}e^{-(s-s_1)})=  O(((s-s_1)+1)e^{-(s-s_1)}).
\eeq
By Lemma \ref{Lemma 2.4}, we can estimate $\rho(\gamma(s))=O(e^{-(s-s_1)})$ for $s\ge s_1$. 
Substituting (\ref{eq2.25bis}), (\ref{eq2.25bisbis}) and  (\ref{eq2.25bisbisbis}),  we have
$$
\begin{array}{ll}
\ds\rho^{-1}\bar\nabla^2\rho( I, II) =O(\rho^2\sqrt{1-\frac{\phi(s)^2}{|\nabla_+ r|^2_{g^+}}})=  O(((s-s_1)+1)e^{-3(s-s_1)})\\
\ds\rho^{-1}\bar\nabla^2\rho( II, II) =O(\rho^2(1-\frac{\phi(s)^2}{|\nabla_+ r|^2_{g^+}}))=  O(((t-t_1)+1)e^{-4(s-s_1)})
\end{array}
$$
so that we established the claim
(\ref{crossterm}).

\vskip .2in

Gathering (\ref{eqmid}), (\ref{crossterm}) and (\ref{eq2.25bisbisbis}), we have
\beq \label{eqmid2}
\begin{array}{ll}
\phi'(s)+\phi(s)^2-1&=-\rho^{-1}\bar\nabla^2\rho(\nabla_+ t , \nabla_+ t)+ |  \bar\nabla \rho|^2_{\bar g}-1\\
&=  -\rho^{-1}\bar\nabla^2\rho(I, I)-2\rho^{-1}\bar\nabla^2\rho( I, II) \,-\rho^{-1}\bar\nabla^2\rho( II, II) + |  \bar\nabla \rho|^2_{\bar g}-1\\
 & = -\rho^{-1}\bar\nabla^2\rho(I, I)+ |  \bar\nabla \rho|^2_{\bar g}-1 + O(((s-s_1)+1)e^{-3(s-s_1)})\\
& = - \rho^{-1}{\frac {|\phi (t)|^2}{|\nabla_+ r|^2 }} \bar\nabla^2\rho (\nabla_+ r, \nabla_+ r) + |  \bar\nabla \rho|^2_{\bar g}-1 + O(((s-s_1)+1)e^{-3(s-s_1)})\\
& = - \rho {\frac {|\phi (t)|^2}{|\nabla_+ r|^2 }} \bar\nabla^2\rho(\bar \nabla \rho, \bar \nabla\rho) +|  \bar\nabla \rho|^2_{\bar g}-1  + O(((s-s_1)+1)e^{-3(s-s_1)})\\
&=-\rho\bar\nabla^2\rho(\bar\nabla\rho, \bar\nabla\rho)+|  \bar\nabla \rho|^2_{\bar g}-1+O(((s-s_1)+1)e^{-3(s-s_1)})\\
\end{array}
\eeq

\vskip .2in

We now use our assumption that  $\rho^2g^+$ is some $C^{2,\alpha}$ almost geodesic compactification, thus
\beq
|  \bar\nabla \rho|^2-1=b(\rho,x_2,\cdots, x_d)
\eeq
is some $C^{2,\alpha}$ function  and $\bar \nabla^2\rho(\bar\nabla\rho, \bar\nabla\rho)$ is some $C^{1,\alpha}$ function such that on the boundary $\p X$
\beq
b=\bar \nabla b=0
\eeq
We calculate the $C^{1, \alpha}$ term
$-\rho\bar \nabla^2\rho(\bar\nabla\rho, \bar\nabla\rho)+ | \bar \nabla \rho|^2_{\bar g}-1$ such that $-\rho\bar \nabla^2\rho(\bar\nabla\rho, \bar\nabla\rho)+ |  \bar \nabla \rho|^2_{\bar g}-1= O(\rho^{2+\alpha})$. Thus we reach the conclusion that 
\beq
\phi'(s)+\phi(s)^2-1=O(e^{-(2+\alpha)(s-s_1)}),
\eeq
and the desired inequality (\ref{eq2.13bis}) in Lemma \ref{Lemma 2.5}\\

As before, we have
\beq
v'(s)=1-(1+v)^2+O(e^{-(2+\alpha)(s-s_1)})\ge -2v+O(e^{-(2+\alpha)(s-s_1)})
\eeq
so that
\beq
v(s)\ge -v(s_1)e^{-2(s-s_1)}+O(e^{-(2+\alpha)(s-s_1)})
\eeq
Therefore, the desired result (\ref{eq2.13bisbis}) and  (\ref{eq2.13bisbisbis}) in Lemma \ref{Lemma 2.5} follow.
\end{proof}

We observe that as a consequence of step 2, we 
also have: \\

{\it Step 3.}  Claim: $\bar \nabla u$ is bounded.\\

Using Lemma \ref{Lemma 2.5}, it follows from (\ref{eq2.9}) that
$$
\nabla_+ t= (1+O(e^{-2(s-s_1)}))\nabla_+ r+ O(e^{-(s-s_1)})\nu=(1+O(\rho)) \nabla_+ r+ O(\rho)\nu
$$
where $\nu $ est unit vector with respect to the metric $g^+$ and $\perp \nabla_+ r$. This gives 
$$
|\nabla_+  u|_{g^+}=  O(\rho)
$$
which yields the desired claim $|\bar \nabla  u|_{\bar g}=  O(1)$.\\

Because of the cut loci,  the distance function $t$ is not regular. Hence it is necessary to understand the fine structure of the set of cut loci and the
behavior of the distance function near the cut loci to do analysis. For  a fixed point $p\in X$, we denote by $C_p$ (resp. $Q_p$) the cut loci (resp. the set of conjugate points) with respect to $p$ (resp. the set of conjugate points). We set $A_p:= C_p\setminus Q_p$ the set of non-conjugate cut loci. We divide $A_p$ into two parts: $N_p$ the normal cut loci is the set of conjugate points from which there are exactly two minimal geodesics connecting to $p$ and realizing the distance to $p$ in $(X, g^+)$, and $L_p:= A_p\setminus N_p$ the set of rest conjugate points. Hence we have the disjoint decomposition
$$
C_p=Q_p\bigcup L_p\bigcup N_p
$$
Recall $X$ is a complete non-compact manifold. Hence, $X\setminus C_p$  is always diffeomorphic to the Euclidean space $\R^d$. Therefore any connected component of $C_p$ extends to the infinity, unless $p$ is a pole, that is, $\exp_p$ is   diffeomorphic from $T_p X$ to $X$. By the results in \cite{IT, O}, we have the following statement.

\begin{lemm}
\label{Lemma 2.5.1} Assume $(X,g)$ is a complete connected $C^\infty$ Riemannian manifold of dimension $d\ge 2$. For any given $p\in X$,  we have 
\begin{itemize}
\item The closed set  $Q_p\bigcup L_p$   is of Hausdorff dimension no more than $d-2$.
\item The set $N_p$ of normal cut loci consists of possibly countably many disjoint smooth hypersurfaces in $X$. Moreover, at each normal cut locus $q\in N_p$, there is a small open neighborhood $U$ of $q$ such that $U\cap C_p = U\cap N_p$  is a piece of smooth hypersurface in $X$.
\end{itemize}
\end{lemm}

Now, we consider the CCE metric $(X,g^+)$ with a given point $p\in X$.  We can use $[0,\delta]\times \mathbb{R}^{d-1}$ as the adapted coordinates of $\bar X_\delta$  along the flow $\frac{\nabla_{ g }\rho}{|\nabla_{ g }\rho|}$ as before. Let $\hat{p}\in \mathbb{R}^{d-1}$ the projection of $p$ on $\p X=\mathbb{R}^{d-1}$. Given $a>0$, we denote $\hat{B}(p,a):= [0,1]\times \bar B_{\mathbb{R}^{d-1}}(\hat{p},a)$ where $\bar B_{\mathbb{R}^{d-1}}(\hat{p},a)$ is the close ball of center $\hat{p}$ and radius $A$ in the euclidean space $\mathbb{R}^{d-1}$. Given a large number $A>0$ and $r\in (0,+\infty)$, let us denote
$$
\gamma_{r,a}=\hat{B}(p,a)\bigcap (\Sigma_r \bigcap C_p)=( \hat{B}(p,a) \cap\Sigma_r  \cap(Q_p\cup L_p ))\bigcup (  \hat{B}(p,a) \cap\Sigma_r \cap N_p)=\gamma_{r,a}^{QL}\bigcup \gamma_{r,a}^{N}
$$
 %In the following, the geodesic balls are always related to the metric $\bar g$ if there is no confusion.  
We set 
$$
\gamma_{r}=\bigcup_{a>0}\gamma_{r,a},\; \gamma_{r}^{QL}=\bigcup_{a>0} \gamma_{r,a}^{QL},\;   \gamma_{r}^{N}=\bigcup_{a>0} \gamma_{r,a}^{N}
$$

We define a mapping $\Delta: X_1\to  \p X\times (0,+\infty)=(0,+\infty)\times \R^{d-1} $
\beq
\Theta(x)=(r(x), \pi(x))=(r(x), (x_2,\cdots, x_d))
\eeq
where $\pi:  X_1\to \p X$ is the projection along $\rho$ (or $r$) foliation.

By the known results of Eilenberg inequality due to Federer \cite[Theorem 2.10.25]{Federer} and also the results in \cite[Lemmas 5.2 and 5.3]{LQS}, we have the following 

\begin{lemm}
\label{Lemma 2.5.2} Assume $(X,g^+)$ is a $C^{2,\alpha}$ with $\alpha\in (0,1)$ CCE  of dimension $d\ge 4$ and $g=\rho^2g^+$ is the adapted metric. For any given $p\in X$,  we have 
\begin{enumerate}
\item For all $t$ large, the set $\Gamma_t\cap X_{\rho_1}$ is a Lipschitz graph over $\pi(\Gamma_t\cap X_{\rho_1})\subset \p X$;
\item When $t$ is sufficiently large, the outward angle of the corner at the normal cut locus on $\Gamma_t$ is always less than $\pi$;
\item For almost every $r>0$, $\gamma^N_r$ has the locally finite $d-2$ dimensional Hausdorff measure ${\mathcal H}^{d-2}$ and consists of possibly countably many disjoint smooth hypersurfaces in $\Sigma_r$.  Moreover, in such case, at each point on $\gamma^N_r$, one has
$$
\frac{\p t}{\p n^+}+ \frac{\p t}{\p n^-}\ge 0
$$
\item  For almost every $r>0$, the $d-2$ dimensional Hausdorff measure on $\gamma_{r}^{QL}$ vanishes
$$
{\mathcal H}^{d-2}(\gamma_{r}^{QL})=0
$$
\end{enumerate}
\end{lemm}
\begin{proof}
The results in 1) and 2)  and the key observation-- the second part of 3) are obtained in \cite[Lemmas 5.2 and 5.3]{LQS}.  It follows from Lemma \ref{Lemma 2.5.1} that $Q_p\bigcup L_p$   has  locally finite $d-2$ dimensional Hausdorff measure and the set $N_p$ of normal cut loci consists of possibly countably many disjoint smooth hypersurfaces in $X$. Using the Eilenberg inequality \cite[Theorem 2.10.25]{Federer}, for almost every $r$,  $\gamma_{r}^{QL}$ has locally finite $d-3$ dimensional Hausdorff measure and $\gamma_{r}^{N}$ has locally finite $d-2$ dimensional Hausdorff measure. Hence,  ${\mathcal H}^{d-2}(\gamma_{r}^{QL})=0$ for almost every $r$. We state that $\Sigma_r$ is a $C^{2,\alpha}$ foliation so that we could get the rest result in 3). In fact, we could write locally $\Sigma_r$ as a graph $r=f(x_2,\cdots,x_d)$ or $x_2=f(r,x_3,\cdots, x_d)$ (we can change the index if necessary). In the latter case, the result is obvious. In the first case, by Sard's Lemma, the image of critical points by $f$ has the$1$-dimensional  Hausdorff measure zero. Hence, the desired result follows. Therefore, we finish the proof.
\end{proof}

{\it Step 4.} 
We denote 
$$
\bar\partial \hat{B}(p,a):=\{y\in \hat{B}(p,a), \|\pi(y)-\hat{p}\|=a\}
$$ 
the lateral boundary of the cylinder $ \hat{B}(p,a)$. It is clear that for any $y^1\in \bar\partial \hat{B}(p,a)$ we have
$$
d_g(y^1, p)\le d_g(y^1, (0,\pi(y^1)))+ d_g( (0,\pi(y^1)), (0,\hat{p}))+d_g( (0,\hat{p}), p)\le a+O(1)
$$
On the other hand, by  Theorem \ref{curvaturedecay}
$$
\frac{a^{1/(1+\varepsilon/2)}}{d((0,\pi(y^1)), (0,\hat{p}))}=o(1)
$$
Hence, the following {\bf estimate} holds %{\color{red} {Yuxin, this seems to me the crucial point of the lengthy estimates, if so, can we mark it??}}
\beq
a^{1/(1+\varepsilon/2)}\le d_g(y^1, p)\le 2a
\eeq
provided $a$ is sufficiently large. Then $s_1\ge 2(1-\beta)\ln a$ for all $\beta>\frac{\varepsilon}{1+\varepsilon/2}$. In fact, we assume $M=\max_{s\in [0,s_1]} \rho(\gamma(s))$. If $M\le a^{1-\beta}$, then
$s_1\ge a^{-1+\beta} d_{g}(p,y^1)=a^{\beta-\frac{\varepsilon}{1+\varepsilon/2}}\ge 2\ln a$ provided $a$ is sufficiently large. Otherwise, $M\ge a^{1-\beta}$. Let $\bar s \in (0,s_1)$ such that $\rho(\gamma(\bar s))=M$. We denote $y2=\gamma(\bar s)$.  We could use $\rho$ as parameter of $\gamma(s)$ thus 
$\bar s\le \int_1^M \frac{1}{\rho}d\rho\le (1-\beta)\ln a$ since $|\bar\nabla \rho|\le 1$. With the similar argument, $s_1-\bar s\le (1-\beta)\ln a$. Thus, we prove the claim. Let $\hat{p}$ be the orthogonal projection of $p$ on the boundary $\R^{d-1}:= \p X$ with respect to the  adapted metric $g$. More precisely, with the same arguments, we can prove $s_1\ge \frac{2}{1+\varepsilon/2}\ln a-\ln 2\ln a$. Recall $\Psi= e^{\frac{(d-3)}{2}u}$ so that 
$$0\le \Psi(x)= e^{\frac{(d-3)}{2}(u(x)-u(\gamma(s_1))} e^{\frac{(d-3)}{2}u(\gamma(s_1))}\le C  e^{\frac{(d-3)}{2}u(\gamma(s_1))} \le C\frac{(2\ln a)^{(d-3)/2}}{a^{(d-3)/(1+\varepsilon/2)}}$$ 
on $\Sigma_r\cap\p  \hat B(p,a)$. Also, it follows from Step 3 we have  on $\Sigma_r\cap\p \hat B(p,a)$
\beq
|\bar\nabla\Psi(x)|= \left|\frac{d-3}{2}\Psi(x)\bar\nabla u(x)\right| \le C\frac{(2\ln a)^{(d-3)/2}}{a^{(d-3)/(1+\varepsilon/2)}}
\eeq
As a consequence, we have
\begin{lemm} \label{Lemma2.8}
Assume $d\ge 5$. We have for almost every $r>-\ln \rho_1$ sufficiently large
$$\oint_{\p (\hat{B}(p,A)\cap \Sigma_r)\setminus C_p}  \Psi\frac{\p \Psi}{\p n}= O(\frac{(2\ln a)^{2(d-3)}}{a^{(2(d-3)/(1+\varepsilon/2))-(d-2)}})$$ 
where $n$ is the outwards unit normal vector.
\end{lemm}
\begin{proof}
We remark that $(\Sigma_r, \bar g_r)$ converges to $(\mathbb{R}^{d-1}, g_{\mathbb{R}^{d-1}})$ in $C^{2,\alpha}$ topology. Hence, for $r$ large, there holds
$$
Vol(\p (\hat{B}(p,A)\cap \Sigma_r))=O(a^{d-2})
$$
which yields
$$
\oint_{\p (\hat{B}(p,A)\cap \Sigma_r)\setminus C_p}  \Psi\frac{\p \Psi}{\p n}=O(\frac{(2\ln a)^{2(d-3)}}{a^{(2(d-3)/(1+\varepsilon/2))-(d-2)}}).
$$

\end{proof}

{\it Step 5.}  By the Gauss-Codazzi equation, we have
\begin{lemm}
On $X_1\bigcap (X\setminus C_p)$, there holds
\label{Lemma 2.9}
\beq
\label{eq2.34}
R_{\tilde g_r}\le \frac{(d-1)(d-2)}{1-e^{-2t}}+\frac{d-2}2e^{2t}(1-\frac{g^+(\nabla_+ r, \nabla_+ t)^2}{|\nabla_+ r|^2_{g^+}})(1-\nabla_+^2 t(\nu,\nu))+e^{2t}O(e^{-4r})
\eeq
for some unit vector $\nu\perp \nabla_+ t$.
\end{lemm}
\begin{proof}
Recall on the $\Sigma_r$, we have the induced metric $g^+_r$ (resp. $\tilde g_r$) from $g^+$ (resp.  $\tilde g=e^{-2t}g^+$). Let $\{e_1,\cdots, e_{d-1}\}$ be an orthonormal basis on $(\Sigma_r, g^+_r)$ so that $\{\tilde e_1,\cdots, \tilde e_{d-1}\}= \{e^t e_1,\cdots, e^t e_{d-1}\}$  is an orthonormal basis on $(\Sigma_r, \tilde g_r)$. And $N=e^t\nabla_+ r/(2|\nabla_+ r|_{g^+})$ is the unit normal to the level set $(\Sigma_r, \tilde g_r)$. Let $\tilde \nabla$ be the Levi-Civita connection with respect to $\tilde g$. We could calculate the second fundamental form with respect to the metric $\tilde g$ on $(\Sigma_r, \tilde g_r)$ 
$$
II[\tilde g] (\tilde e_i, \tilde e_j)=\tilde g(\tilde\nabla_{\tilde e_i} N, \tilde e_j)=e^t/(2|\nabla_+ r|_{g^+})(\nabla_+^2 r(e_i,e_j)-g^+(\nabla_+ t,\nabla_+ r)\delta_{ij})
$$
which implies by Lemmas \ref{Lemma2.2} and \ref{Lemma 2.5} 
\beq
\label{secondform}
h_{ij}=II[\tilde g] (\tilde e_i, \tilde e_j)=\tilde g(\tilde\nabla_{\tilde e_i} N, \tilde e_j)=e^t/(2|\nabla_+ r|_{g^+})(\nabla_+^2 r(e_i,e_j)-g^+(\nabla_+ t,\nabla_+ r)\delta_{ij}) =O( e^{t-2r})
\eeq
Here we use the fact in Lemma \ref{Lemma 2.4} that $r(\gamma(s))= t(\gamma(s))-s_1+O(1)$ for all $s\ge s_1$.  Recall the Gauss-Codazzi equation
\beq
\label{Gauss-Codazzi}
R_{\tilde g_r}=R[\tilde g]-2Ric[\tilde g](N,N)+(\sum_i h_{ii})^2-\sum_{ij} (h_{ij})^2
\eeq
Together with the equations (\ref{eq2.21}), (\ref{eq2.22}) and (\ref{secondform}), we infer
\beq
\label{scalarest}
R_{\tilde g_r}=e^{2t}\frac{d-2}{2}(\triangle_+ t-(d-1)+(1-\frac{g^+(\nabla_+ r, \nabla_+ t)^2}{|\nabla_+ r|^2})-\nabla_+^2 t(\frac{\nabla_+ r}{|\nabla_+ r|},\frac{\nabla_+ r}{|\nabla_+ r|_{g^+}})+O(e^{-4r}))
\eeq
By the Laplace comparison theorem, we have
\beq
\label{eq2.38}
\triangle_+ t\le (d-1)\coth t
\eeq
On the other hand, we have the decomposition
\beq
\label{eq2.39}
\frac{\nabla_+ r}{|\nabla_+ r|_{g^+}}= \frac{g^+(\nabla_+ r, \nabla_+ t)}{|\nabla_+ r|_{g^+} }\nabla_+ t+ \sqrt{1-\frac{g^+(\nabla_+ r, \nabla_+ t)^2}{|\nabla_+ r|_{g^+}^2}}\nu
\eeq
for some unit vector $\nu\perp \nabla_+ t$ and 
\beq
\label{eq2.40}
\nabla_+^2 t(\nabla_+ t, \cdot)=0
\eeq
since $|\nabla_+ t|_{g^+}\equiv 1$.  Gathering (\ref{scalarest}) to (\ref{eq2.40}), we deduce
\beq
\label{scalarest1}
R_{\tilde g_r}\le \frac{(d-1)(d-2)}{1-e^{-2t}}+ e^{2t}\frac{d-2}{2}((1-\frac{g^+(\nabla_+ r, \nabla_+ t)^2}{|\nabla_+ r|_{g^+}^2})(1-\nabla_+^2 t(\nu,\nu))+O(e^{-4r}))
\eeq
Finally, we prove the desired result (\ref{eq2.34}).
\end{proof}

{\it Step 6.} %On $\Sigma_r$, the norm of horizontal part of $\nabla t$ is  
We know $\Psi$ is a uniformly bounded and uniformly Lipschitz function on $(X\setminus C_p)\bigcap X_1$ so that it could be extended to a uniformly bounded and uniformly Lipschitz function on $\overline{(X\setminus C_p)\bigcap X_1}=\overline{X_1}$.
\begin{lemm} 
\label{Lemma 2.10}
Given any $a>1$ large and $\eta>0$ small, we have for almost every large $r$ there exists some finite union $B_{r,\eta}$ of balls with radius small than $\eta$ covering $\gamma^{QL}_{r,N}$ such that 
\beq
\label{eq2.42}
\begin{array}{ll}
&\ds  \frac{4(d-2)}{d-3}\int_{\bar B_{\mathbb{R}^{d-1}}(\hat{p},a)} |\bar\nabla \Psi|^2 dvol [g_{\mathbb{R}^{d-1}}]\\
 \le &\ds  \int_{(\Sigma_r\cap \hat{B}(p,a))\setminus (B_{r,\eta}\cup \gamma_{r,A}^N)}R_{\tilde g_r} dvol_{\tilde g_r}+ O(\eta)+ O(\frac{(2\ln a)^{2(d-3)}}{a^{(2(d-3)/(1+\varepsilon/2))-(d-2)}})+o_r(1)
\end{array}
\eeq
and
\beq 
\label{eq2.43}
\displaystyle  \int_{\bar B_{\mathbb{R}^{d-1}}(\hat{p},a)}\Psi^{2(d-1)/(d-3)} dvol [g_{\mathbb{R}^{d-1}}]= vol_{\tilde g_r}(\Sigma_r\cap \hat{B}(p,a))+o_r(1)
\eeq
where $o_r(1)$ is independent of $\eta$ and $o_r(1)\to 0$ as $r\to\infty$.
%where the last $o(1)$ on the right hand means $\lim_{R\to+\infty} o(1)=0$.
\end{lemm}
\begin{proof}
Because $\rho^2 g^+$ is $C^{2,\alpha}$ and its injectivity radii is bounded below, we have $R[\bar g_r]$ converges uniformly to $0$ as $r$ tends to $\infty$. Recall $\tilde g= \Psi^{\frac4{d-3}} \bar g$. By the conformal change, we infer on $(X\setminus C_p)\bigcap X_1$
$$
R[\tilde g_r]= \Psi^{-\frac{d+1}{d-3}}(-\frac{4(d-2)}{d-3}\triangle[\bar g_r] \Psi+ R[\bar g_r] \Psi). 
$$
On the other hand,  by Lemma \ref{Lemma 2.5.2}, for almost every $r>0$, we have 
$$
\mathcal{H}^{d-2}(\gamma_{r,a}^{QL})=0,\;\; \mathcal{H}^{d-2}(\gamma_{r,a}^{N})<\infty.
$$
And  $\gamma_{r,a}^{N}$ consists of possibly countably many disjoint smooth hypersurfaces in $\Sigma_r$. Denote $\eta>0$ some small positive number. Let $B_{r,\eta}=\cup_i B(x_i,r_i)$ be finite the union of the balls with $x_i\in \gamma_{r,A}^{QL}$ covering $\gamma_{r,A}^{QL}$ such that $\sum_i r_i^{d-2}\le \eta$ and $r_i<1$ for all $i$ (in the following, the geodesic balls are always related to the metric $\bar g$ if there is no confusion).  We estimate
$$
\begin{array}{ll}
&\ds\int_{(\hat{B}(p,a)\cap \Sigma_r)\setminus (\cup_i B(x_i,r_i)\cup \gamma_{r,a}^{N})} R[\tilde g_r] dvol[\tilde g_r]\\
=&\ds\int_{(\hat{B}(p,a)\cap \Sigma_r)\setminus (\cup_i B(x_i,r_i)\cup \gamma_{r,a}^{N})} 
\Psi(-\frac{4(d-2)}{d-3}\triangle[\bar g_r] \Psi+ R[\bar g_r] \Psi)dvol[\bar g_r]\\
=& \ds\int_{(\hat{B}(p,a)\cap \Sigma_r)\setminus (\cup_i B(x_i,r_i)\cup \gamma_{r,a}^{N})} \frac{4(d-2)}{d-3}|\bar\nabla \Psi|^2+ R[\bar g_r] \Psi^2dvol[\bar g_r]\\
&\ds +\oint_{\p((\hat{B}(p,a)\cap \Sigma_r)\setminus (\cup_i B(x_i,r_i)\cup \gamma_{r,a}^{N}))} \frac{4(d-2)}{d-3}\Psi(\frac{\p \Psi}{\p n}d\sigma)[\bar g_r]
\end{array}
$$
We write
$$
\begin{array}{ll}
&\ds \oint_{\p((\hat{B}(p,a)\cap \Sigma_r)\setminus (\cup_i B(x_i,r_i)\cup \gamma_{r,a}^{N}))} \Psi(\frac{\p \Psi}{\p n}d\sigma)[\bar g_r]\\
=&\ds \oint_{(\hat{B}(p,a)\cap \Sigma_r)\cap \p (\cup_i B(x_i,r_i))} \Psi(\frac{\p \Psi}{\p n}d\sigma)[\bar g_r]+ \oint_{\p(\hat{B}(p,a)\cap \Sigma_r)\setminus (\cup_i B(x_i,r_i)\cup \gamma_{r,a}^{N})} \Psi(\frac{\p \Psi}{\p n}d\sigma)[\bar g_r]\\
&+ \ds \oint_{\gamma_{r,a}^{N}\setminus (\cup_i B(x_i,r_i))} -\frac{d-3}{2}\Psi^2((\frac{\p t}{\p n^+}+\frac{\p t}{\p n^-})d\sigma)[\bar g_r]
\end{array}
$$
Thanks of Lemmas \ref{Lemma 2.5.2} and \ref{Lemma2.8}, there holds
$$
\begin{array}{ll}
\ds -\frac{d-3}{2}\oint_{\gamma_{r,a}^{N}\setminus (\cup_i B(x_i,r_i))}\Psi^2((\frac{\p t}{\p n^+}+\frac{\p t}{\p n^-})d\sigma)[\bar g_r] \le 0\\
\ds\oint_{\p(\hat{B}(p,a)\cap \Sigma_r)\setminus (\cup_i B(x_i,r_i)\cup \gamma_{r,R}^{N})} \Psi(\frac{\p \Psi}{\p n}d\sigma)[\bar g_r]=O(\frac{(2\ln a)^{2(d-3)}}{a^{(2(d-3)/(1+\varepsilon/2))-(d-2)}})
\end{array}
$$
On the other hand, we have
$$
\begin{array}{ll}
\ds \oint_{(\hat{B}(p,a)\cap \Sigma_r)\cap \p (\cup_i B(x_i,r_i))} \Psi(\frac{\p \Psi}{\p n}d\sigma)[\bar g_r]=O(\eta)\\
\ds\int_{(\hat{B}(p,a)\cap \Sigma_r)\cap (\cup_i B(x_i,r_i))} \frac{4(d-2)}{d-3}|\bar\nabla \Psi|^2=O(\eta)
\end{array}
$$
since $\sum_i r_i^{d-1}\le \sum_i r_i^{d-2}\le \eta$. Also, as  $R[\bar g_r] $ converges uniformly to $0$, we infer
$$
\int_{\hat{B}(p,a)\cap \Sigma_r} | R[\bar g_r]| \Psi^2dvol[\bar g_r]= o_r(1)
$$
Gathering all above estimates, we derive for almost every $r$
$$
\begin{array}{ll}
&\ds\frac{4(d-2)}{d-3}\int_{\hat{B}(p,a)\cap \Sigma_r} |\bar\nabla \Psi|^2  dvol [\bar g_r]\\
 \le&\ds  \int_{(\Sigma_r\cap \hat{B}(p,a))\setminus (B_{r,\eta}\cup \gamma_{r,A}^N)}R_{\tilde g_r} dvol_{\tilde g_r}+ O(\eta)+ O(\frac{(2\ln a)^{2(d-3)}}{a^{(2(d-3)/(1+\varepsilon/2))-(d-2)}})+o_r(1)
 \end{array}
$$
We know $\Psi$ is a uniformly bounded and uniformly Lipschitz function on $\overline{X_1}$ so that $\Psi|_{\Sigma_r}\to \Psi|_{\mathbb{R}^{d-1}}$ in $C^0$ topology and weakly in $H^1$ topology on all compact subset. Hence we get
$$
\begin{array}{ll}
\ds\int_{\bar B_{\mathbb{R}^{d-1}}(\hat{p},a)} |\bar\nabla \Psi|^2dvol [g_{\mathbb{R}^{d-1}}]\le \liminf \int_{\hat{B}(p,a)\cap \Sigma_r} |\bar\nabla \Psi|^2  dvol [\bar g_r],\\
\ds\int_{B(p,a)\cap \R^{d-1}}\Psi^{2(d-1)/(d-3)}dvol [g_{\mathbb{R}^{d-1}}]=  \int_{B(p,a)\cap \Sigma_r}\Psi^{2(d-1)/(d-3)}dvol [\bar g_r]+o(1)
 \end{array}
$$
which yields the desired results (\ref{eq2.42}) and (\ref{eq2.43}). Therefore, we finish the proof.
\end{proof}

{\it Step 7.} \begin{lemm} On $\R^{d-1}$, $\Psi $ does not vanish.\end{lemm}
\begin{proof} 
Let $x\in \hat{B}(p,2)\cap X_{\rho_1/2}$ and $\gamma$ be $g^+$ minimizing geodesic from $p$ to $x$.  From the proof of Lemma \ref{Lemma 2.3}, we have 
$$
d_{g^+}(p,x)=-\log \rho(x)+O(1)=O(1)
$$
Hence $u(x)$ is bounded on $ \hat{B}(p,2)\cap X_{\rho_1/2}$. As  a consequence the Lipschitz function $\Psi(x)= e^{(d-3)u(x)/2}$ is bounded below by  some positive constant on $  B_{\bar g}(p,2)\cap X_{\rho_1/2}$. Hence,  $\Psi $ does not vanish on the boundary $\R^{d-1}$.
\end{proof}

{\it Step 8.}  
\begin{lemm} 
For $p\in X$ fixed before and $a>1$ fixed, we have
\beq
\lim_{r\to \infty}vol_{\tilde g_r}(\Sigma_r\cap \hat{B}(p,a))\le \lim_{t\to \infty}vol_{\tilde g_t}(\Gamma_t)\le vol(\S^{d-1}
\eeq
\end{lemm}
\begin{proof}  
We know $(X,g^+)$ is Einstein. Thus, it is smooth and the exponential map is smooth. As same arguments as in Lemma \ref{Lemma 2.5.1}, for almost every $t>0$, 
$$
\mathcal{H}^{d-1}(\Gamma_t\cap C_p)=0.
$$
In view of  Lemma \ref{Lemma 2.5}, we can estimate for any $q\in X_{\rho_1}\setminus C_p$,  we infer (see also \cite[Lemma 6.1]{Dutta})
$$
\|\pi_*\tilde g_r- \pi_*\bar g_t\|_{g_{\mathbb{R}^{d-1}}}\le C e^{-r}
$$
Hence, thanks of Lemma  \ref{Lemma 2.10}, we derive
\beq 
\begin{array}{ll}
&\displaystyle  \int_{\bar B_{\mathbb{R}^{d-1}}(\hat{p},a)}\Psi^{2(d-1)/(d-3)} dvol [g_{\mathbb{R}^{d-1}}]\\
=&\ds\lim_{r\to \infty} vol_{\tilde g_r}(\Sigma_r\cap \hat{B}(p,a))=\lim_{r\to \infty} vol_{\tilde g_r}((\Sigma_r \setminus C_p)\cap \hat{B}(p,a))\\
=& \ds\lim_{t\to \infty}vol_{\tilde g_t}((\Gamma_t \setminus C_p)\cap \hat{B}(p,a)\cap X_{\rho_1})\le \liminf_{t\to \infty}vol_{\tilde g_t}(\Gamma_t )
\end{array}
\eeq
On the other hand, by Bishop's comparison theorem, for all $t>0$, there holds
$$
\frac{vol[g^+](\Gamma_t)}{vol[g^{\mathbb{H}}](\Gamma_t)}\le \frac{vol[g^+](B_{g^+}(p,t)}{vol[g^{\mathbb{H}}](B_{g^{\mathbb{H}}}(0,t))}\le 1
$$
Here we use $\Gamma_t$ to denote geodesic sphere for the various metrics if there is no confusion. Therefore, 
$$
\lim_{t\to \infty}vol_{\tilde g_t}(\Gamma_t)\le \lim_{t\to \infty} e^{(d-1)t}vol[g^{\mathbb{H}}](\Gamma_t)=\lim_{t\to \infty}\frac{e^{(d-1)t}}{\sinh^{d-1}t} vol(\S^{d-1})=vol(\S^{d-1})
$$
Finally, we prove the Lemma.
\end{proof}

{\it Step 9.}  
\begin{lemm}
\label{Lemma 2.13}
For fixed $p\in X$ and any fixed large $a>1$, then for any small $\eta>0$ and for almost all large $r>0$ with $\mathcal {H}^{d-2}(\gamma_{r,a}^{QL})=0$, we have
\beq
\int_{(\Sigma_r\cap \hat{B}(p,a))\setminus(B_{r,\eta}\cup\gamma_{r,a}^{N})}R_{\tilde g_r} dvol_{\tilde g_r}\le (d-1)(d-2) vol_{\tilde g_r}((\Sigma_r\cap \hat{B}(p,a))\setminus(B_{r,\eta}\cup\gamma_{r,a}^{N}))+o_r(1)
\eeq
where $B_{r,\eta}$ is given in Lemma \ref{Lemma 2.10}, $o_r(1)$ is independent of $\eta$ and $o_r(1)\to 0$ as $r\to\infty$.
\end{lemm}
\begin{proof} The proof is similar as that one \cite[Theorem 6.8]{LQS}. We sketch the proof.  Let $\gamma(s)$ be a minimizing geodesic w.r.t. $g^+$ connecting $p$ to $q\in \hat{B}(p,A)\cap X_{\rho_1/2}$. Hence the corresponding $s_1$ with $\gamma(s_1)\in \p X_{\rho_1}$ is bounded above by some constant $C>0$ depending on $A$ so that $u$ is bounded by Lemma \ref{Lemma 2.4}.  Moreover, there exists $C>0$ such that $\forall s>s_1$ we have $d_{\bar g}(p, \gamma(s))\le C$.  We denote $S=\nabla_+^2 t$ the shape operator of the geodesic sphere along $\gamma(s)$. $S$ satisfies the following Riccati equation
\beq
\label{Riccati1}
 \nabla_{\nabla_+ t} S+ S^2=-(R_{\nabla t})[g^+]
\eeq
 where $ (R_{\nabla t})[g^+](\nu)=(R(\nu,\nabla t)\nabla t[g^+]$ for any tangent vector $\nu$ orthogonal to $\gamma'( s)$. Therefore, for all $s>s_1$ and the principal curvature $\mu$, we have 
 $$
 1-Ce^{-2(s-s_1)}\le \mu'(s)+  \mu^2(s)\le 1+ Ce^{-2(s-s_1)}
 $$
 which gives
\beq
\label{Riccati2}
1-C_0e^{-2s}\le \mu'(s)+  \mu^2(s)\le 1+ C_0e^{-2s}
\eeq 
where $C_0>0$ is some constant depending on $R$. Let us denote by $\mu_m$ and $\mu_M$ the smallest and the biggest principal curvature of the geodesic sphere $\Gamma_t$ respectively.  Remark $H=\triangle t$ is the mean curvature of the geodesic sphere $\Gamma_t$. The main idea in \cite{LQS} is the following: when $H$ is close to $d-1$,  all principal curvatures are close to 1; when $H$ is away from $d-1$, the integral on the such set $\int (1-\mu_m) \le o_r(1)$.  The key ingredients are based on a careful analysis of Riccati equations (\ref{Riccati1}) and (\ref{Riccati2}). The proof is divided in several steps. We set $t_0=s_0$ and $t_1=s_1$.\\
{\it Step a. } Given a large $a>0$ and $t_0$ with $(\Gamma_{t_0}\setminus C_p)\cap \hat{B}(p,a)\cap X_{\rho_1}\neq \emptyset$, there exists some $C>0$ such that for all $q\in (\Gamma_{t_0}\setminus C_p)\cap \hat{B}(p,a)\cap X_{\rho_1}$ we have (see \cite[Lemma6.1]{LQS})
\beq
\label{eq2.49}
\mu_M(q)\le C
\eeq
Moreover $\forall t>t_0$, we have for all $q\in (\Gamma_{t}\setminus C_p)\cap \hat{B}(p,A)\cap X_{\rho_1}$
\beq
\label{eq2.50}
\mu_M(q)\le 1+C(t+1)e^{-2t}
\eeq
{\it Step b. } For fixed $p\in X$ and any fixed large $a>1$, given any $1>\kappa>0$ small fixed, we set for large $r>-\log \rho_1$
\beq
U_{r,a}^\kappa=\{q\in (\Sigma_r\setminus C_p)\cap \hat{B}(p,a)\;: \;H(q)\le (d-1)(1-\kappa)\}
\eeq
Then we have $\ds \int_{U_{r,a}^\kappa} d vol[\tilde g_r]\to 0$ as $r\to\infty$ (see \cite[Proposition 6.5]{LQS}). \\
{\it Step c. } For fixed $p\in X$ and any fixed large $a>1$, given any $1>\kappa>0$ small fixed, there exists a constant $C>0$ such that for $q\in (\Gamma_t\setminus  C_p)\cap \hat{B}(p,a)$ with $H(q)\le -2(d-1)$ and $\gamma$ being the minimizing geodesic from $p$ to $q$, there holds
\beq
\label{eq2.52}
Hdvol[\tilde g_t]\ge Ce^{-\frac{\kappa t}{4}}
\eeq
provided $t>t_\kappa+1$ where $t_\kappa$ is some positive number depending on $\kappa, a$. (see \cite[Proposition 6.7]{LQS})\\
{\it Step d. }  Recall $u$ is bounded on $\hat{B}(p,a)\setminus C_p$. By Lemma \ref{Lemma 2.5}, it follows from (\ref{eq2.34}) on $\hat{B}(p,a)\setminus C_p$
\beq
\label{eq2.53}
R_{\tilde g_r}\le (d-1)(d-2)+C(1-\nabla_+^2 t(\nu,\nu))+o_r(1)
\eeq
where $C>0$ is some constant independent of $r$ and $o_r(1)\to 0$ as $r\to \infty$. On the other hand, by (\ref{eq2.50}),  we have on $\hat{B}(p,a)\setminus C_p$
\beq
\label{eq2.54}
\nabla_+^2 t(\nu,\nu)\ge \mu_m\ge H-C
\eeq
Fix a small $\kappa>0$ and $\eta>0$. We divide the set $(\Sigma_r\cap \hat{B}(p,a))\setminus(B_{r,\eta}\cup\gamma_{r,a}^{N})$ into three subsets
$$
(\Sigma_r\cap \hat{B}(p,a))\setminus(B_{r,\eta}\cup\gamma_{r,a}^{N}):= A_{\kappa, \eta, 1}\cup A_{\kappa, \eta, 2}\cup A_{\kappa, \eta, 3}
$$
where 
$$
\begin{array}{ll}
A_{\kappa, \eta, 1}:=\{q\in (\Sigma_r\cap \hat{B}(p,a))\setminus(B_{r,\eta}\cup\gamma_{r,a}^{N}): \; H(q)>(d-1)(1-\kappa)\}\\
A_{\kappa, \eta, 2}:=\{q\in (\Sigma_r\cap \hat{B}(p,a))\setminus(B_{r,\eta}\cup\gamma_{r,a}^{N}): \; -2(d-1)<H(q)\le (d-1)(1-\kappa)\}\\
A_{\kappa, \eta, 2}:=\{q\in (\Sigma_r\cap \hat{B}(p,a))\setminus(B_{r,\eta}\cup\gamma_{r,a}^{N}): \; H(q)\le -2(d-1)\}
\end{array}
$$
On $A_{\kappa, \eta, 1}$, we infer by  (\ref{eq2.50}) 
$$
1-\nabla_+^2 t(\nu,\nu)\le 1-\mu_m\le 1-H+(d-2)\mu_M \le o_r(1)+O(\kappa). 
$$
On  $A_{\kappa, \eta, 2}\cup A_{\kappa, \eta, 3}$, we use (\ref{eq2.54}). Thus
we can estimate by (\ref{eq2.52}) to (\ref{eq2.54})
$$
\begin{array}{ll}
\ds\int_{A_{\kappa, \eta, 1}}  R_{\tilde g_r} dvol_{\tilde g_r}&\le \ds(d-1)(d-2) \int_{A_{\kappa, \eta, 1}}   dvol_{\tilde g_r} +O(\kappa)+ o_r(1)\\
\ds\int_{A_{\kappa, \eta, 2}}  R_{\tilde g_r} dvol_{\tilde g_r}&\le \ds(d-1)(d-2) \int_{A_{\kappa, \eta, 2}}   dvol_{\tilde g_r} +Cvol[\tilde g_r](U_{r,A}^\kappa) + o_r(1)\\
&=\ds (d-1)(d-2) \int_{A_{\kappa, \eta, 2}}   dvol_{\tilde g_r} + o_r(1)\\
\ds\int_{A_{\kappa, \eta, 3}}  R_{\tilde g_r} dvol_{\tilde g_r}&\le \ds4(d-1)(d-2) \int_{A_{\kappa, \eta, 3}}   dvol_{\tilde g_r} +Cvol[\tilde g_r](U_{r,A}^\kappa) + o_r(1)\\
&+\ds \int_{A_{\kappa, \eta, 3}}   Hdvol_{\tilde g_r}\\
&=\ds (d-1)(d-2) \int_{A_{\kappa, \eta, 3}}   dvol_{\tilde g_r} +O(\kappa)+ o_r(1)
\end{array}
$$
which yields 
$$
\int_{(\Sigma_r\cap \hat{B}(p,a))\setminus(B_{r,\eta}\cup\gamma_{r,a}^{N})}R_{\tilde g_r} dvol_{\tilde g_r}\le (d-1)(d-2) \int_{(\Sigma_r\cap \hat{B}(p,a))\setminus(B_{r,\eta}\cup\gamma_{r,a}^{N})}dvol_{\tilde g_r}+O(\kappa)+ o_r(1)
$$
As $\kappa>0$ could be chosen sufficiently small, we get the desired result. Hence, Lemma is proved.
\end{proof}

{\it Step 10.}
\begin{lemm}
$g$ is flat
\end{lemm}
\begin{proof}
Gathering Lemmas (\ref{Lemma 2.10}) to (\ref{Lemma 2.13}), for fixed $R>1$ and $\eta>0$ there holds for almost every $r$ large
\beq
\begin{array}{ll}
&\frac{\ds \frac{4(d-2)}{d-3}\int_{\bar B_{\R^{d-1}}(\hat{p},a)} |\bar\nabla \Psi |^2dvol [g_{\mathbb{R}^{d-1}}]+ O(\eta)+ O(\frac{(2\ln a)^{2(d-3)}}{a^{(2(d-3)/(1+\varepsilon/2))-(d-2)}})+o_r(1)}{ \ds (\int_{\bar B_{\R^{d-1}}(\hat{p},a)}\Psi^{2(d-1)/(d-3)}dvol [g_{\mathbb{R}^{d-1}}]+ O(\eta) +o_r(1))^{\frac{d-3}{d-1}}}\\
\le &\frac{\ds  \int_{(\Sigma_r\cap \hat{B}(p,a))\setminus (B_{r,\eta}\cup \gamma_{r,a}^N)}R_{\tilde g_r} dvol_{\tilde g_r}}{\ds vol_{\tilde g_r}(\ds(\Sigma_r\cap \hat{B}(p,a))\setminus (B_{r,\eta}\cup \gamma_{r,a}^N))}\\
=&(d-1)(d-2) vol_{\tilde g_r}((\Sigma_r\cap \hat{B}(p,a))\setminus(B_{r,\eta}\cup\gamma_{r,a}^{N}))^{\frac{2}{d-1}}+o_r(1)\\
=&(d-1)(d-2) vol_{\tilde g_r}(\Sigma_r\cap \hat{B}(p,A))^{\frac{2}{d-1}}+o_r(1)+o(\eta)\\
\le &(d-1)(d-2) (vol(\S^{d-1}))^{\frac{2}{d-1}}+o_r(1)+o(\eta)
\end{array}
\eeq
Taking the limit successively as $r\to\infty$, $\eta\to 0$ and $a\to \infty$, we derive
\beq
\int_{\R^{d-1}} |\bar\nabla \Psi |^2dvol [g_{\mathbb{R}^{d-1}}]/( \int_{\R^{d-1}}\Psi^{2(d-1)/(d-3)}dvol [g_{\mathbb{R}^{d-1}}])^{\frac{d-3}{d-1}}\le \frac{(d-1)(d-3)}{4}vol(\S^{d-1})^{\frac{2}{d-1}}
\eeq
Recall for any $v\in H^1(\R^{d-1})$
\beq
\int_{\R^{d-1}} |\bar\nabla v |^2/( \int_{\R^{d-1}}|v|^{2(d-1)/(d-3)})^{\frac{d-3}{d-1}}\ge \frac{(d-1)(d-3)}{4}vol(\S^{d-1})^{\frac{2}{d-1}}
\eeq
Thus, $\Psi$ is extremal function for the Sobolev embedding. Hence $g$ is flat. To see this, we have
\beq
\frac{(d-1)(d-3)}{4}vol(\S^{d-1})^{\frac{2}{d-1}}= \frac{(d-1)(d-3)}{4}vol(\S^{d-1})^{\frac{2}{d-1}}\lim_{t\to \infty}\frac{vol_{ g^+}(\Gamma_t)^{\frac{2}{d-1}}}{vol_{ g^\H}(\Gamma_t)^{\frac{2}{d-1}}}
\eeq
which implies for all $t>0$ by Bishop's comparison theorem
\beq
vol_{g^+}(B(p,t))=vol_{g^\H}(B(p,t))
\eeq
Hence $g^+$ is hyperbolic space form. %Hence $g$ is flat.
\end{proof}

\begin{RK}

\begin{enumerate} 
Under more careful computations,
    \item one can improve the upper bound for $\varepsilon$ in Theorem \ref{Liouville}.
    \item One can replace condition (3) in Theorem \ref{Liouville}, by a weaker assumption that is,
    $$
    \limsup_{x\to\infty}|Ric(x)|\rho(x)^2<\frac{1}{2}.
    $$
\end{enumerate}
\end{RK}


\begin{thebibliography}{99}



\bibitem{Anderson0} M. Anderson, {\sl Convergence and rigidity of manifolds under Ricci curvature bounds}, \newblock{Invent. Math.}, 
(1990) 102 (2), 429 - 445.

\bibitem{anderson1} M. Anderson, {\sl  Einstein metrics with prescribed conformal infinity on 4-manifolds}, \newblock  {Geom. Funct. Anal. }18 (2008), 305 - 366.



\bibitem{anderson5} M. Anderson, {\sl L2 curvature and volume renormalization of the AHE metrics on 4-manifolds}, \newblock {\em Math. Res. Lett.}, 8 (2001) 171 - 188.


%\bibitem{AKKLT} M. Anderson, A. Katsuda, Y. Kurylev, M. Lassas, and M. Taylor, {\sl  Boundary regularity for the Ricci equation, geometric convergence, and Gelfand's inverse boundary problem},  Invent. Math., 158 (2): 261 - 321, 2004.


 



\bibitem{Biquard} O. Biquard, {\sl  Continuation unique \`a partir de l'infini conforme pour les m\'etriques d'Einstein}, \newblock {\em Math.  Res. Lett.} 15 (2008), 1091-1099.

%\bibitem{Biquard1} O. Biquard, {\sl Einstein deformations of hyperbolic metrics}, surveys in differential geometry: essays on Einstein manifolds, 235-246, Surv. Differ. Geom., 6, Int. Press, Boston, MA, 1999.

%\bibitem{BiquardHerzlich} O. Biquard and M. Herzlich,  {\sl  Analyse sur un demi-espace hyperbolique et poly-homog\'en\'eit\'e locale},  \newblock {\em Calc. Var. P.D.E. 51} (2014), 813-848.

%\bibitem{BB} H.P.Boas and R.P. Boas,  {\sl  Short proofs of three theorems on harmonic function}, \newblock {\em Proceeding of the AMS} 102 (1988), 906-908.






\bibitem{casechang} J. Case and S.Y. A. Chang,  {\sl  On fractional GJMS operators},  \newblock {\em Comm. Pure Appl. Math.} 69 (2016), 
1017-1061.


\bibitem{CG} S.Y. A. Chang and Y. Ge, {\sl  Compactness of conformally compact Einstein manifolds in dimension 4},   \newblock {\em Adv. Math.} 340 (2018), 588-652.

\bibitem{CG1} S.Y. A. Chang and Y. Ge, {\sl  On the problem of filling by a Poincar\'e-Einstein metric in dimension 4 },  preprint, arXiv:2509.18430.

\bibitem{CGJQ} S.Y. A. Chang, Y. Ge, S. Jin and J. Qing, {\sl Perturbation compactness and uniqueness for a class of conformally compact Einstein manifolds}, \newblock {\em Adv. Nonlinear Stud.} 24 (2024), no. 1, 247-78.

\bibitem{CGQ} S.Y. A. Chang, Y. Ge and J. Qing, {\sl  Compactness of conformally compact Einstein manifolds in dimension 4 II}, \newblock {\em Adv. Math.} 373(2020), 107325.




\bibitem{CYa} S.Y. A. Chang and  R. Yang,  {\sl   On a class of non-local operators in conformal geometry}, preprint.



%\bibitem{CGT} J. Cheeger, M. Gromov and M. Taylor, {\sl  Finite propagation speed, kernel estimates for functions of the Laplace operator, and the geometry of complete Riemannian manifolds}, J. Differential Geom. 17 (1982), 15-53.


%\bibitem{CDLS} P. T. Chru\'sciel, E. Delay, J. M. Lee, D. N. Skinner: Boundary regularity of           conformally compact Einstein metrics, \emph{Journal of Differential Geometry}, \textbf{69(1)}, 111--136 (2004)


\bibitem{Dutta} Satyaki Dutta and Mohammad Javaheri,\newblock {\em Rigidity of conformally compact manifolds with the round sphere as the conformal infinity}, \newblock {\em Adv. Math.} 224 (2010), 525-538.

\bibitem{Federer} H. Federer,   {\sl Geometric measure theory}, \newblock { Grundlehren der mathematischen Wissenschaften} 153, Ed. Springer, (1996).

%\bibitem{FG} C. Fefferman and C. R. Graham, \newblock {\em  $Q$-curvature and Poncar\'e metrics}, \newblock {Math. Res. Lett.} 9 (2002), 139 -151.

\bibitem{FG12} C. Fefferman and C. R. Graham; \newblock {\em  The ambient metric},  Annals of Mathematics Studies, 178, Princeton University Press, Princeton, (2012).

%\bibitem{GT} D. Gilbarg and N. Trudinger   {\sl Elliptic partial differential equations of second order}, \newblock { Grundlehren der mathematischen Wissenschaften} 224, Ed. Springer, (2001).


\bibitem{G00} C. R. Graham,
\newblock {\em  Volume and Area renormalizations for conformally compact Einstein metrics}, \newblock The Proceedings of the 19th Winter School "Geometry and Physics" (Srn\`{i}, 1999). 
\newblock Rend. Circ. Mat. Palermo 63  (2000) 31-42.

\bibitem{GH} C. R. Graham and K. Hirachi,
\newblock {\em  The ambient obstruction tensor and $Q$-curvature,}, \newblock In AdS/CFT correspondence: Einstein metrics and their conformal boundaries, 
\newblock volume 8 of IRMA lec. Math. Theor. Phys., pages 59-71, Eur.Math. Soc., Z\"urich, 2005.



\bibitem{GL} C. R. Graham and J. Lee, {\em Einstein metrics with prescribed conformal infinity on the ball.} \newblock {Adv. Math.} 87 
(1991), no. 2, 186 - 225.




\bibitem{GZ} C. R. Graham and M. Zworski,
\newblock {\em  Scattering matrix in conformal geometry}, \newblock {Invent. Math.} 152 (2003), 89-118.


%\bibitem{Helli} D. W. Helliwell, {\sl Boundary regularity for conformally compact Einstein metrics in even dimensions}, \newblock {Communications in Partial Differential Equations}, 33(5), (2008), 842 - 880. 

\bibitem{IT}  J. Itoh and  M. Tanaka, {\em  The dimension of a cut locus on a smooth Riemannian manifold}, \newblock {Tohoku Math. J.} (2), 50,
(1998), 571-575.

%\bibitem{Kodani} S. Kodani, {\sl  Convergence theorem for Riemannian manifolds with boundary}, Compositio Math., 75(2):171 - 192, 1990.

%\bibitem{Knox} K.Knox,  {\sl  A compactness theorem for riemannian manifolds with boundary and applications} arXiv:1211.6210 [math.DG], 


\bibitem{Lee1} J. Lee, {\sl Fredholm operators and Einstein metrics on conformally compact manifolds}, \newblock{Mem. Amer. Math. 
Soc.} 183 (2006), no. 864, vi+83 pp.

\bibitem{Sanghoon Lee} S.Lee, 
\emph{Liouville type theorems on the hyperbolic space}, \newblock{ Calc. Var. Partial Differential Equations} 64 (2025), no. 7, Paper No. 217.

\bibitem{LeeWang} S.Lee and F. Wang
\emph{Rigidity of Poincar\'e-Einstein manifolds with flat Euclidean conformal infinity}, preprint, arXiv:2503.06062.

\bibitem{LQS} G. Li, J. Qing and Y. Shi, \newblock {\em  Cap phenomena and curvature estimates for conformally compact Einstein manifolds },
\newblock {Trans. Amer. Math. Soc. 369 (2017), no. 6, 4385 - 4413.} 



\bibitem{Mald-1} J. Maldacena, The large N limit of superconformal field theories and supergravity, 
\textit{ Adv. Theo. Math. Phy.} {\textbf  2} (1998) 231-252, hep-th/9711200

\bibitem{Mald-2} J. Maldacena, TASI 2003 Lectures on AdS/CFT, hep-th/0309246.

\bibitem{Mald} J. Maldaceana, {\sl Einstein gravity from conformal gravity},  arXiv:1105.5632.

\bibitem{MM} R. Mazzeo and R.  Melrose, {\sl  Meromorphic extension of the resolvent on complete spaces with asymptotically constant negative curvature}, \newblock { J. Funct. Anal.} 113 (1991),  25-45 .

\bibitem{O} V. Ozols, {\em Cut Loci in Riemannian Manifolds},  \newblock {Tohoku Math. J.} 26 (1974), 219 -227.

%\bibitem{Perales} R. Perales, {\em A survey on the convergence of manifolds with boundary}, preprint: arXiv:1310.0850.


\bibitem{Q} J. Qing, {\sl   On the rigidity for conformally compact Einstein manifolds}, IMRN Volume 2003, Issue 21,  1141-1153.


\bibitem{Wi} E.Witten  Anti de Sitter space and holography, \textit{ Adv.Theor.Math.Phys.}, {\textbf 2} (1998), 253-291


%\bibitem{Wong} J. Wong, {\sl   An extension procedure for manifolds with boundary}, PacificJ.Math., 235(1):173 - 199, 2008.



            










\end{thebibliography}
\end{document}